\documentclass{amsart}

\usepackage{amsmath,amsfonts,amssymb,amsthm,amscd,mathdots}
\usepackage{color}
\usepackage{tikz}

\newtheorem{thm}{Theorem}[section]
\newtheorem{prop}[thm]{Proposition}

\theoremstyle{remark}
\newtheorem{rem}{Remark}[section]

\theoremstyle{remark}
\newtheorem{ex}{Example}[section]

\numberwithin{equation}{section}

\newcommand{\be}{\begin{equation}}  \newcommand{\ee}{\end{equation}}
\newcommand{\nn}{\nonumber}

\begin{document}


\def\C{{\mathbb C}}
\def\R{{\mathbb R}}
\def\F{{\mathcal F}}
\def\T{{\mathcal T}}
\def\E{{\mathcal E}}
\def\Z{{\mathbb Z}}
\def\N{{\mathbb N}}

\def\L{{\mathcal L}}
\def\M{{\mathcal M}}
\def\P{{\mathcal P}}

\def\l{\lambda}
\def\sgn{\mbox{\rm sgn}}
\def\varen{\varepsilon}

\title{Geometric RSK and the Toda lattice}

\author{Neil O'Connell}
\address{Mathematics Institute, University of Warwick, Coventry CV4 7AL, UK}
\email{n.m.o-connell@warwick.ac.uk}

\maketitle

\begin{abstract}
We relate a continuous-time version of the geometric RSK correspondence to the Toda lattice, in a way
which can be viewed as a semi-classical limit of a recent result by the author which relates the 
continuous-time geometric RSK mapping, with Brownian motion as input, to the quantum Toda lattice.
\end{abstract}

\tableofcontents

\section{Introduction}

The geometric RSK correspondence is a geometric lifting of the classical RSK correspondence.
It was introduced by A.N. Kirillov~\cite{ki} and further studied by Noumi and Yamada~\cite{ny}.  
There is also a 
continuous-time version of the geometric RSK mapping, which was introduced in~\cite{oc03} and substantially 
developed in~\cite{bbo1,bbo2} in the context of Littelmann's path model.  In this setting, an important role 
is a played by a mapping $\Pi_n$ (defined in the next section) which takes as input a continuous path 
$\eta(t),\ t \ge 0$ in $\R^n$ with $\eta(0)=0$ and returns a path $\Pi_n\eta(t),\ t>0$, also in $\R^n$.
In the paper~\cite{noc} it was shown that, if $\eta(t)=\sqrt{\epsilon} B(t)+t\l$, where $\l\in\R^n$ and
$B$ is a standard Brownian motion in $\R^n$, then $\Pi_n\eta(t),\ t>0$ is a diffusion process in $\R^n$
with infinitesimal generator given in terms of the Hamiltonian of the open quantum Toda chain with
$n$ particles.  The aim of the present paper is to understand this result from the point of view of the
classical Toda lattice.  

By considering the semiclassical limit ($\epsilon\to0$) of this result, it can be seen (heuristically) from 
a result of Givental~\cite{g} (at least in the case $\l=0$) that, if $\eta(t)=t\l$, 
then $\Pi_n\eta(t),\ t>0$ should define a solution to the classical Toda flow (with opposite sign), and 
indeed this is the case.  We will show it directly in the classical ($\epsilon=0$) setting by considering
the continuous-time geometric RSK mapping with deterministic input.  The approach is very elementary 
and mostly self-contained.  Starting with the definition of
the geometric RSK mapping, we soon arrive at familiar objects in the general theory of the Toda lattice,
thus providing further insight into the results of~\cite{noc} from an integrable systems point 
of view. 

The main conclusion is that there is a precise sense in which
\be\label{equation}
\mbox{classical Toda + noise = quantum Toda.}
\ee
This statement requires some qualification, however.  First, we consider both the classical and quantum
system in imaginary time.  For the classical system this means that the potential has a minus sign,
that is, the Hamiltonian is given by
$$\frac12\sum_{i=1}^n p_i^2 - \sum_{i=1}^{n-1} e^{x_{i+1}-x_i}.$$
For the quantum system, it means we consider normalised Hamiltonian
$$\L_\lambda = -\frac12 \psi_\l(x)^{-1} \left( H+\epsilon\sum_i\l_i^2\right) \psi_\l(x)
= \frac\epsilon2 \Delta + \epsilon \nabla\log \psi_\l \cdot \nabla ,$$
where
$$H=-\epsilon\Delta+\frac2\epsilon\sum_{i=1}^{n-1} e^{x_{i+1}-x_i}$$
is the Hamiltonian of the quantum Toda lattice
and $\psi_\l$ is a particular eigenfunction of $H$ known as a (class one) $GL(n,\R)$-Whittaker function.
Moreover, in the `equation' \eqref{equation}, the noise must be added in a very particular way.
It is not simply a random perturbation of the classical Toda flow, but rather a random perturbation 
of a particular construction of the Toda flow which is closely related to the geometric RSK correspondence.
Within this construction, the perturbation is simply: 
\begin{center} {\em Add noise to the constants of motion}.\end{center}
It will be interesting to investigate to what extent this relation extends to other integrable systems.

Along the way we observe the following curious fact.  Consideration of the geometric
RSK mapping with Brownian motion as input gives rise to a particular stochastic dynamics on triangles
(the analogue of Gelfand-Tsetlin patterns in this setting), as discussed in~\cite{noc}.  In that paper, 
another quite different stochastic dynamics on triangles was also considered and shown to have the 
same fixed-time distributions, and to bear the same relation to the quantum Toda lattice.  This latter 
dynamics can be interpreted as a geometric lifting of Warren's process~\cite{w}, which in turn can be 
interpreted as a continuous version of a {\em shuffling algorithm} which has played an important role in the 
random tilings literature~\cite{n}.  Similar dynamics on Gelfand-Tsetlin patterns, constructed using a
general prescription of Diaconis and Fill~\cite{df}, have been studied by Borodin and co-workers, see 
for example~\cite{bc,bf}.   It turns out that, in the semi-classical limit we consider in this paper, 
both the `RSK type' and `shuffling type' of dynamics are equivalent.

The outline of the paper is as follows.  In the next section we recall some relevant background material on 
factorisations of matrices.  In Section \ref{s-grsk}, we define and recall some basic properties of the continuous-time 
geometric RSK mapping.  In Section \ref{s-wfqt}, we recall some facts about Whittaker functions and the quantum
Toda lattice, which continue to play a role in the classical setting.  In Section \ref{grsk-bm}, we recall some of 
the main results of~\cite{noc} and in Section \ref{scl} we briefly outline, at a heuristic level, what happens to 
these results in the semiclassical limit.  In Section \ref{sec-toda}, we recall some basic definitions and properties of 
the opposite sign Toda lattice and, in Section \ref{flows}, we formulate and prove the main results of the paper.

\bigskip

\noindent{\bf Acknowledgements.}  Thanks to Mark Adler, Percy Deift, Nick Ercolani, Govind Menon
and Pierre van Moerbeke for helpful discussions.  This research was supported in part by 
EPSRC grant number EP/I014829/1. 

\section{Preliminaries}

Let $G=GL(n,\C)$ and denote by $B, N$ (resp. $B_-,N_-$) the subgroups of upper (resp. lower) triangular and 
uni-triangular matrices in $G$.  Throughout this paper, an important role will be played by the totally positive part, 
which we will denote by $\P$, of the double Bruhat cell $B\cap B_-\bar w_0 B_-$, where $\bar w_0$ is a representative 
in $G$ of the longest element $w_0\in S_n$, as in~\cite{fz,bfz}.  Concretely,  
\be\label{P}
\P=\{b\in B: \Delta^m_k(b)> 0,\ 1\le k\le m\le n\}
\ee
where
\be\label{delta}
\Delta^m_k(b)=\det \Big[ b_{ij}\Big]_{1\le i\le k,\ m-k+1\le j\le m}.
\ee
In the following we adopt the convention that $\Delta^m_0(b)=1$.
The quantities $\Delta^m_k(b),\ 1\le k\le m\le n$ uniquely determine $b\in\P$, as follows.
(See, for example,~\cite{bfz} or \cite[Proposition 1.5]{ny}.)
For $a\in\C^k$, define
$$\epsilon^{k}(a)=\begin{pmatrix} a_1 & 1  &0  &\dots& 0\\
0 & a_2 & 1  &\dots& 0\\
\vdots &  & \ddots &  & \vdots\\
&&&a_{k-1} &1\\
0&&\dots&&a_k \end{pmatrix},$$
and denote by $I_k$ the identity matrix of dimension $k$.
For $1\le m\le n$ and $w\in\C^{n-m+1}$, define
$$E_m(w)=\begin{pmatrix} I_{m-1} & 0 \\ 0 & \epsilon^{n-m+1}(w) \end{pmatrix}.$$
\begin{prop}\label{ny} 
Each element $b\in \P$ can be represented uniquely as a product of the form
$$b=E_1(w^1)\dots E_n(w^n)$$ where $w^m\in(\R_{> 0})^{n-m+1}$ for each $m$.
The $w^m$ are given by
$$w^m_1=\Delta^m_1(b);\qquad w^m_i=\frac{\Delta^{m+i-1}_i(b)\Delta^{m+i-2}_{i-2}(b)}
{\Delta^{m+i-1}_{i-1}(b)\Delta^{m+i-2}_{i-1}(b)},\qquad 1<i\le n-m+1.$$
\end{prop}
This provides a natural (Gelfand-Tsetlin) parameterization of $\P$ by the set of {\em triangles} 
\be\label{T}
\T=\{X=(x_i^m)\in\R^{n(n+1)/2}:\ 1\le i\le m\le n\},
\ee
obtained by setting
\be\label{det1}
x^m_1+\cdots+x^m_k=\log \Delta^m_k(b) .
\ee
We will denote the corresponding bijective mapping by $f:\P\to\T$.
Note that if we write $X=(x^m_i)=f(b)$, the $w^m$ of Proposition~\ref{ny} are given, in terms of the $x^m_i$, 
by
$$w^m_1=e^{x^m_1};\qquad w^m_i=e^{x^{m+i-1}_i-x^{m+i-2}_{i-1}},\qquad 1<i\le n-m+1.$$

The Weyl group associated with $G$ is the symmetric group $S_n$.  
Each element $w\in S_n$ has a representative $\bar w\in G$ defined as follows.  
Denote the standard generators for
$\mathfrak{gl}_n$ by $h_i$, $e_i$ and $f_i$.  For example, for $n=3$,
$$h_1=\begin{pmatrix}
1& 0 &0\\
0 & 0 & 0 \\
0 & 0 & 0 \end{pmatrix},\qquad 
h_2 =\begin{pmatrix}
0& 0 &0\\
0 & 1 & 0 \\
0 & 0 & 0 \end{pmatrix},\qquad 
h_3 = \begin{pmatrix}
0& 0 &0\\
0 & 0 & 0 \\
0 & 0 & 1 \end{pmatrix},
$$
$$e_1=\begin{pmatrix}
0& 1 &0\\
0 & 0 & 0 \\
0 & 0 & 0 \end{pmatrix},\quad 
e_2 =\begin{pmatrix}
0& 0 &0\\
0 & 0 & 1 \\
0 & 0 & 0 \end{pmatrix},\quad f_1=\begin{pmatrix}
0& 0 &0\\
1 & 0 & 0 \\
0 & 0 & 0 \end{pmatrix},\quad 
f_2 =\begin{pmatrix}
0& 0 &0\\
0 & 0 & 0 \\
0 & 1 & 0 \end{pmatrix}.
$$

For adjacent transpositions $s_i=(i,i+1)$, define 
$$\bar{s}_i=\exp(-e_i)\exp(f_i)\exp(-e_i)=(I-e_i)(I+f_i)(I-e_i).$$  
In other words, 
$\bar s_i=\varphi_i\begin{pmatrix}0&-1\\1&0\end{pmatrix}$
where $\varphi_i$ is the natural embedding of $SL(2)$ into $GL(n)$
given by $h_i$, $e_i$ and $f_i$.  For example, when $n=3$,
$$\bar s_1 = \begin{pmatrix}
0& -1 &0\\
1 & 0 & 0 \\
0 & 0 & 1 \end{pmatrix},\qquad 
\bar s_2 =
\begin{pmatrix}
1& 0 &0\\
0 & 0 & -1 \\
0 & 1 &0 \end{pmatrix}.$$
Now let $w=s_{i_1}\ldots s_{i_r}$ be a reduced decomposition and define
$\bar{w}=\bar s_{i_1}\ldots \bar s_{i_r}$.  Note that $\overline{uv}=\bar u\bar v$
whenever $l(uv)=l(u)+l(v)$.
Denote the longest element of $S_n$ by
$$w_0=\left(\begin{array}{cccc} 1 & 2 & \cdots & n \\ n& n-1 & \cdots & 1\end{array}\right) .$$ 
For $n=2$, $ w_0=s_1$ and $$\bar  w_0=\bar s_1 =  \begin{pmatrix}
0 & -1  \\
1 & 0 \end{pmatrix} .$$
For $n=3$, $ w_0=s_1s_2s_1=s_2s_1s_2$ is represented by
$$
\bar  w_0= \bar s_1\bar s_2\bar s_1=\bar s_2\bar s_1\bar s_2 = \begin{pmatrix}
0& 0 &1\\
0 & -1 & 0 \\
1 & 0 & 0 \end{pmatrix}  .$$

Denote the elementary lower uni-triangular Jacobi matrices by 
$$l_i(a)=I_n+a f_i, \qquad 1\le i\le n-1.$$
These matrices play a central role in parameterising of the set $(N_-)_{>0}$ of 
totally positive lower triangular matrices, see for example~\cite{bfz}.
For $u\in\C^m$, $1\le m< n$, define
$$L_m(u)=l_m(u_m)l_{m-1}(u_{m-1})\ldots l_1(u_1).$$
\begin{prop}\label{lus}  Each $L\in (N_-)_{>0}$ can be written uniquely as a product
$$L=L_1(u^1) L_2(u^2) \dots L_{n-1}(u^{n-1}) ,$$
where $u^m\in (\R_{>0})^m$ for each $m$.
\end{prop}
The next proposition is due to Berenstein, Fomin and Zelevinsky~\cite{bfz}.
\begin{prop}\label{umi} Let $b\in\P$.  Then $b \bar  w_0$ has a Gauss decomposition 
$b \bar  w_0=LDU$ where 
$$D_{ii}=\frac{\Delta^n_i(b)}{\Delta^n_{i-1}(b)},\qquad 1\le i\le n,$$
and $L\in (N_-)_{>0}$ is given by
$$L=L_1(u^1) L_2(u^2) \dots L_{n-1}(u^{n-1}) ,$$
$$u^m_i=\frac{\Delta^m_{i-1}(b)\Delta^{m+1}_{i+1}(b)}{\Delta^m_i(b)\Delta^{m+1}_i(b)},\qquad 1\le i\le m<n.$$
\end{prop}
Note that, if $X=(x^m_i)=f(b)$, then the $D_{ii}$ and $u^m_i$ of Proposition~\ref{umi} are given by
$$D_{ii}=e^{x^n_i},\qquad 1\le i\le n$$
and
$$u^m_i=e^{x^{m+1}_{i+1}-x^{m}_i},\qquad 1\le i\le m<n.$$

\section{Geometric RSK in continuous time}\label{s-grsk}

In this section we recall the definition and some basic properties of the continuous-time 
geometric RSK mapping.  Many of the results of this section are essentially contained in the
papers~\cite{bbo1,bbo2}, see also~\cite{noc,noc1}.  For completeness
we include direct proofs of all the main statements, which are adapted to the present setting.

Let $\eta :[0,\infty)\to\R^n$ be a continuous path with $\eta(0)=0$.
Denote the coordinates of $\eta$ by $\eta_1,\ldots,\eta_n$, so that
$$\eta(t)=(\eta_1(t),\ldots,\eta_n(t)),\qquad t>0.$$
For $1\le i<j\le n$, set
$$\Omega_{ij}(t)=\{0<s_{j-1}<\ldots < s_i< t\}.$$
In the following, it will be helpful to think of elements $\phi\in\Omega_{ij}(t)$ as `down-right paths' 
in the semi-lattice $\R\times\Z$ starting at $(0,j)$ and ending at $(t,i)$, as shown in Figure 1. 
\begin{figure}
\begin{center}
\begin{tikzpicture}[scale=.7]
\draw [help lines] (-1,-1) -- (7,-1); 
\draw [help lines] (-1,0) -- (7,0); 
\draw [help lines] (-1,1) -- (7,1); 
\draw [help lines] (-1,2) -- (7,2); 
\draw [help lines] (-1,3) -- (7,3); 
\draw [help lines] (-1,4) -- (7,4); 
\draw [help lines] (-1,5) -- (7,5); 

\node at (-1,-1.5) {$0$};
\node at (7,-1.5) {$t$};

\draw [very thick] (-1,4)--(0.5,4);
\draw [dotted] (.5,4)--(.5,-1);
\node at (.5,-1.5) {$s_{j-1}$};

\draw [very thick] (0.5,3)--(1.5,3);
\draw [dotted] (1.5,3)--(1.5,-1);
\node at (1.5,-1.5) {$s_{j-2}$};

\draw [very thick] (1.5,2)--(3.5,2);
\draw [dotted] (3.5,2)--(3.5,-1);
\node at (3.5,-1.5) {$s_{i+1}$};

\draw [very thick] (3.5,1)--(5,1);
\draw [dotted] (5,1)--(5,-1);
\node at (5,-1.5) {$s_{i}$};

\draw [very thick] (5,0)--(7,0);

\node at (-1.5,4) {$j$};
\node at (7.5,0) {$i$};

\end{tikzpicture}
\end{center}
\caption{A down-right path $\phi=(s_{j-1},\ldots,s_i)\in\Omega_{ij}(t)$.}
\end{figure}
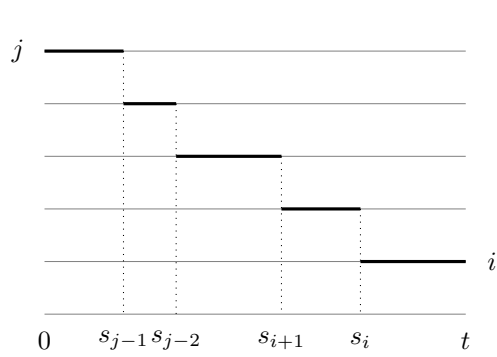
Write $d\phi=ds_i\cdots ds_{j-1}$ for the Euclidean measure
on $\Omega_{ij}(t)$.  For $\phi=(s_{j-1},\ldots,s_i)\in\Omega_{ij}(t)$, we define
$$E_\eta(\phi)=\eta_j(s_{j-1})+\eta_{j-1}(s_{j-2})-\eta_{j-1}(s_{j-1})+\cdots+\eta_i(t)-\eta_i(s_i).$$
For $1\le i<j\le n$ and $t\ge 0$, set
\be\label{def-b}
b_{ii}(t)=e^{\eta_i(t)},\qquad b_{ij}(t)=\int_{\Omega_{ij}(t)} e^{E_\eta(\phi)} d\phi.
\ee
Setting $b_{ij}=0$ for $i>j$, this defines a path in the subgroup $B$ of upper triangular matrices in $GL(n,\C)$.  
If $\eta$ is smooth, then $b(t)=(b_{ij}(t))$, $t\ge 0$, satisfies the evolution equation
\be\label{bdot-eta}
\dot b = \epsilon(\dot\eta) b,
\ee
with initial condition $b(0)=I_n$, where $\epsilon(\l)=\epsilon_\l$ is defined for $\l\in\C^n$ by
\be\epsilon_\l=\begin{pmatrix} \l_1 & 1  &0  &\dots& 0\\
0 & \l_2 & 1  &\dots& 0\\
\vdots &  & \ddots &  & \vdots\\
&&&\l_{n-1} &1\\
0&&\dots&&\l_n \end{pmatrix}.\ee
For $n=1$ we define $\epsilon_\l=\l_1$.

The next proposition is a special case of \cite[Proposition 3.11]{bbo1}.
It shows that $b(t)\in\P$ for each $t>0$ and also makes clear the connection with
Kirillov's original definition of the geometric RSK mapping in terms on non-intersecting lattice paths.
For completeness we include a direct proof which is adapted to the present setting.
\begin{prop}\label{kmg}
\be\label{kmgf}
\Delta^m_k(b(t)) = \int_{\Omega^m_k(t)} e^{E_\eta(\phi_1)+\cdots+E_\eta(\phi_k)}d\phi_1\ldots d\phi_k,
\ee
where the integral is with respect to the Euclidean measure on the set $\Omega^m_k(t)$ of 
$k$-tuples of non-intersecting down-right paths $\phi_1,\ldots,\phi_k$ starting at $(0,m-k+1),\ldots,(0,m)$,
respectively, and ending at $(t,1),\ldots,(t,k)$.  
\end{prop}
\begin{proof}  
This is a straightforward variation of the Karlin-McGregor / Lindstr\"om-Gessel-Viennot formula,
and is proved by a standard path-switching argument, see for example~\cite[Section 1.2]{joc}.
Let $$Y(t)=(Y_1(t),\ldots,Y_k(t)),\ t\ge 0$$ be a collection of independent, unit-rate Poisson
processes started at positions $(1,\ldots,k)$.  Then $Y$ is a continuous time Markov chain with state
space $\N^k$.  Denote the transition probabilities of $Y$ by
$$p_t(y,y')=P(Y(s+t)=y'|\ Y(s)=y),\qquad s,t\ge 0.$$
For $\sigma\in S_k$ and $y\in\N^k$, write $\sigma y=(y_{\sigma(1)},\ldots,y_{\sigma(k)})$.
We note that $Y$ enjoys the strong Markov property and its law is $S_k$-invariant, that is, 
$p_t(\sigma y,\sigma y')=p_t(y,y')$ for all $\sigma\in S_k$.  Let
$$T_i=\inf\{t>0:\ Y_i(t)=Y_{i+1}(t)\},\qquad 1\le i\le n-1,$$
and $T=\min_i T_i$.  Fix $t>0$ and let $Z$ be an integrable, measurable function of $(Y(r),\ 0\le r\le t)$
which is invariant under the substitutions $Y \to \tilde Y^{(i)}$, where for each $i$,
$$\tilde Y^{(i)} (r)=\begin{cases} Y(r) & r\le T_i\\ s_i Y(r) & r >T_i\end{cases} ,$$
with $s_i$ denoting the adjacent transposition $(i,i+1)$.

We will first show that, for $y\in\N^k$ with $y_1<y_2<\cdots<y_k$,
\be\label{kmgz}
E [Z;T>t;Y(t)=y]=\sum_{\sigma\in S_k} \sgn(\sigma) E[ Z; Y(t)=\sigma y].
\ee
Note that, since $E [Z;T>t;Y(t)=\sigma y]=0$ unless $\sigma$ is the identity, this is equivalent to
$$\sum_{\sigma\in S_k} \sgn(\sigma) E[ Z; T\le t; Y(t)=\sigma y]=0,$$
or indeed
$$\sum_i \sum_{\sigma\in S_k} \sgn(\sigma) E[ Z; T=T_i\le t; Y(t)=\sigma y]=0.$$
It therefore suffices to show that, for each $i$,
$$\sum_{\sigma\in S_k} \sgn(\sigma) E[ Z; T=T_i\le t; Y(t)=\sigma y]=0.$$
Now fix $i$ and write $\tilde Y=\tilde Y^{(i)}$.  Then, by the strong Markov property
and $S_k$-invariance of $Y$, $\tilde Y$ has the same law as $Y$.  Moreover,
$$T_j=\inf\{t>0:\ \tilde Y_j(t)=\tilde Y_{j+1}(t)\},\qquad 1\le j\le n-1.$$
Thus,
\begin{align*}
E[ Z; T=T_i\le t; Y(t)=\sigma y]&=E[ Z; T=T_i\le t; \tilde Y(t)=s_i \sigma y]\\
&=E[ Z; T=T_i\le t; Y(t)=s_i \sigma y],
\end{align*}
and we conclude that
\begin{align*}
\sum_{\sigma\in S_k} \sgn(\sigma) & E[ Z; T=T_i\le t; Y(t)=\sigma y]\\
& = \sum_{\sigma\in S_k} \sgn(\sigma)  E[ Z; T=T_i\le t; Y(t)=s_i \sigma y]\\
&= - \sum_{\sigma\in S_k} \sgn(s_i \sigma)  E[ Z; T=T_i\le t; Y(t)=s_i \sigma y]\\
&= - \sum_{\sigma\in S_k} \sgn(\sigma) E[ Z; T=T_i\le t; Y(t)=\sigma y],
\end{align*}
as required.  

To see that \eqref{kmgz} implies the formula \eqref{kmgf}, we take
$$Z=e^{E_{\eta}(\hat Y_1)+\cdots+E_{\eta}(\hat Y_k)},$$
where $\hat Y(s)=Y(t-s)$, and $y=(m-k+1,\ldots,m)$.  Then, using the fact that a collection
of independent, identically distributed exponential random variables, conditioned on the value
of their sum is uniformly distributed in the corresponding simplex, we can write:
\begin{align*}
E [Z;&T>t;Y(t)=y]\\
&= P(Y(t)=y) \left|\Omega_{k,m}(t)\right|^{-k} 
\int_{\Omega^m_k(t)} e^{E_\eta(\phi_1)+\cdots+E_\eta(\phi_k)}d\phi_1\ldots d\phi_k,
\end{align*}
and
\begin{align*}
\sum_{\sigma\in S_k} \sgn(\sigma) & E[ Z; Y(t)=\sigma y]
= \det \left[ E[ e^{E_\eta(\hat Y_i)}; Y_i(t)=y_j ]\right]_{1\le i,j\le k} \\
&= \det \left[ P(Y_i(t)=y_j) \left|\Omega_{i,m-k+j}(t)\right|^{-1} b_{i,m-k+j}(t)\right]_{1\le i,j\le k} .
\end{align*}
Now, for each $i\le j$,
$$P(Y_i(t)=j)=e^{-t}\frac{t^{j-i}}{(j-i)!}=e^{-t} \left|\Omega_{ij}(t)\right|,$$
hence
$$P(Y_i(t)=y_j) \left|\Omega_{i,m-k+j}(t)\right|^{-1} = e^{-t}$$
and
$$P(Y(t)=y) \left|\Omega_{k,m}(t)\right|^{-k} = e^{-kt},$$
which concludes the proof.
\end{proof}

In particular, $X(t)=(x^m_i(t))=f(b(t)),\ t>0$ defines a path in the set of triangles $\T$.
The mapping $\Pi:\eta\mapsto (X(t),\ t>0)$ was introduced and studied in the papers~\cite{bbo1,bbo2}
and can be thought of as a continuous time version of the geometric RSK correspondence introduced by 
A. N. Kirillov~\cite{ki}.  It also appeared, in a different form (see below) in the paper~\cite{oc03}.
For readers familiar with the usual RSK correspondence, for each fixed $t>0$, the path $(\eta(s),\ 0\le s\le t)$ 
should be interpreted as the input `word', the triangle $X(t)=(x^m_i(t),\ 1\le i\le m\le n)$ as the `$P$-tableau',
the path $(x^n(s),\ 0<s\le t)$ as the `$Q$-tableau' and the vector $x^n(t)$ as their common `shape'.  

The mapping $\Pi:\eta\mapsto (X(t),\ t>0)$ defined above admits the following alternative formulation which
is, in fact, equivalent to the original definition given in~\cite{oc03}.
For $i=1,\ldots,n-1$, and continuous $\eta :(0,\infty)\to\R^n$, define 
$$(P_i \eta) (t) = \eta(t) + \left( \log \int_0^t e^{\eta_{i+1}(s)-\eta_i(s)} ds\right) (e_i-e_{i+1}) , $$
where $e_1,\ldots,e_n$ denote the standard basis vectors in $\R^n$.
Let $\Pi_1$ denote the identity mapping ($\Pi_1\eta=\eta$) and, for $2\le m\le n$, define
$$\Pi_m = P_1\circ\cdots\circ P_{m-1} \circ \Pi_{m-1}.$$
Now, for continuous $\eta:[0,\infty)\to\R^n$ with $\eta(0)=0$, define $X(t)=(x^m_i(t))$, $t>0$, by
\be\label{grsk}
x^m_i(t)=(\Pi_m\eta)_i(t),\qquad \quad 1\le i\le m\le n.
\ee
Then it holds that $\Pi\eta=(X(t),\ t>0)$.
This follows from a more general result~\cite[Theorem 3.5]{bbo1}, 
which states that $(e^{x^n_1},\ldots,e^{x^n_n})$ is the diagonal part in the Gauss decomposition of 
$b\bar{w_0}$, as well as the recursive nature of the construction.  For completeness we will include a
self-contained proof of this fact in the following, see Proposition~\ref{sum} below.

In~\cite{bbo1} it was shown that the $P_i$ satisfy the braid relations, that is
$$P_i P_{i+1} P_i = P_{i+1} P_i P_{i+1},\qquad i=1,\ldots,n-1.$$
It follows that, for each $w\in S_n$,
$$ P_w:=  P_{i_r}\cdots   P_{i_1}$$
is well defined, where $w=s_{i_1}\cdots s_{i_r}$ is any reduced decomposition of $w$
as a product of adjacent transpositions, where $s_i$ denotes the transposition $(i,i+1)$.
The above-defined $\Pi_n$ is in fact $P_{w_0}$, as can be seen using the
reduced decomposition $1\ 21\ 321\ \ldots\ n-1\ldots 21$.

It is a straightforward consequence of \eqref{grsk} that, for smooth $\eta$, the triangle 
$X=(x^m_i)$ evolves according to the dynamics 
\begin{eqnarray}\label{dyn-rsk-pi}
& \dot x^1_1=\dot\eta_1; \\ 
& \dot x^m_1=\dot x^{m-1}_1 + e^{x^m_2-x^{m-1}_1},\qquad \dot x^m_m=\dot\eta_m-e^{x^m_m-x^{m-1}_{m-1}},
\qquad 2\le m\le n;\nn \\ \nn
& \dot x^m_i = \dot x^{m-1}_i + e^{x^m_{i+1}-x^{m-1}_i} - e^{x^m_i-x^{m-1}_{i-1}},\qquad 1< i< m\le n.
\end{eqnarray}
(For details, see Proposition~\ref{S-extend}.)  Note that the initial value is singular.  

We now will explain how, using the evolution equations~\eqref{dyn-rsk-pi}, one can `insert' a path $\eta$
into an arbitrary initial triangle $\xi\in\T$. In the language of RSK, this corresponds to inserting a word 
into an arbitrary initial $P$-tableau.  For a smooth path $\eta$, the dynamic \eqref{dyn-rsk-pi} defines
a flow on triangles which we denote by $S^\eta_t$.  In other words, if we set $X(0)=\xi$ and allow $X(t)$
to evolve according to \eqref{dyn-rsk-pi}, then $X(t)=S^\eta_t\xi$.  Similarly, we denote by $R^\eta_t$
the flow on $\P$ defined by $R^\eta_tb_0=b(t)$, where $b(t)$ is the solution to \eqref{bdot-eta} with initial
condition $b(0)=b_0$.  We will explain shortly how to extend the definitions of the flows $S^\eta_t$ and 
$R^\eta_t$ to continuous paths $\eta$, but first we make a note of the important relation between them.
\begin{prop}\label{sr}  For smooth $\eta:[0,\infty)\to\R^n$ with $\eta(0)=0$,
\be
R^\eta_t = f^{-1} \circ S^\eta_t \circ f.
\ee
\end{prop} 
\begin{proof}
Let $X(0)\in\T$ and, for $t\ge 0$, $X(t)=S^\eta_t X(0)$ and $b(t)=f^{-1}(X(t))$.
Let us write $\l=\dot\eta$.  We are required to show that $\dot b=\epsilon_\l b$.
We will prove this by induction.  Write $\T=\T^n$, $f=f^n$, $\P=\P^n$, $\epsilon=\epsilon^n$ 
to emphasize their dependence on $n$ and for $m<n$ denote by $S^\pi_t$ the flow defined on $\T^m$ 
defined in the same way as above by a smooth path $\pi:[0,\infty)\to\R^m$.  

For $n=1$, $b=e^{x^1_1}$ so $\dot x^1_1=\l_1$ implies $\dot b = \l_1 b$, as required.  

For general $n$, by Proposition~\ref{ny} we can write
$$b=E_1(w^1)\dots E_n(w^n)$$
where 
$$w^m_1=e^{x^m_1};\qquad w^m_i=e^{x^{m+i-1}_i-x^{m+i-2}_{i-1}},\qquad 1<i\le n-m+1.$$
Define $Y=(y^m_i)\in\T^{n-1}$ by $y^m_i=x^{m+1}_i$, $1\le i\le m\le n-1$.  In other words,
$Y$ is the triangle of size $n-1$ obtained from $X$ by removing $x^1_1,\ldots,x^n_n$.
Define a smooth path $\pi:[0,\infty)\to\R^{n-1}$ by setting $\pi(0)=0$ and $\dot\pi=\nu$, where
\be\label{nu}
\nu_1=\l_1+a_1,\quad \nu_2=\l_2-a_1+a_2,\quad \ldots\quad \nu_{n-1}=\l_{n-1}-a_{n-2}+a_{n-1},
\ee
$$a_i=e^{x^{i+1}_{i+1}-x^i_i},\qquad 1\le i\le n-1.$$
Then, by \eqref{dyn-rsk-pi}, $Y(t)=S^\pi_t Y(0)$.  Moreover, from the definition
of the $w^m$ above, we can write
$$b=E_1(w^1)\begin{pmatrix} 1 & 0 & \dots & 0\\
0 & &&\\
\vdots&&b^{n-1}&\\
0&&&\end{pmatrix}$$
where $b^{n-1}=(f^{n-1})^{-1}(Y)$.  
We note that $w^1=(e^{x^1_1},a_1,\ldots,a_{n-1})$ and so, by \eqref{dyn-rsk-pi},
\be\label{vdot}
\dot w^1= (\l_1 w_1^1,(\l_2-\l_1-a_1) w^1_2, \ldots, (\l_n-\l_{n-1}-a_{n-1})w^1_n) .
\ee
By the induction hypothesis, $\dot b^{n-1}=\epsilon^{n-1}(\nu) b^{n-1}$.  Thus
\begin{align*}
\dot b &= \dot E_1(w^1)\begin{pmatrix} 1 & 0 & \dots & 0\\0 & &&\\\vdots&&b^{n-1}&\\0&&&\end{pmatrix} 
+ E_1(w^1) \begin{pmatrix} 0 & 0 & \dots & 0\\0 & &&\\\vdots&&\epsilon^{n-1}(\nu) b^{n-1}&\\0&&&\end{pmatrix} \\
&=\left[ \dot E_1(w^1) + E_1(w^1) \begin{pmatrix} 0 & 0 & \dots & 0\\0 & &&\\\vdots&&\epsilon^{n-1}(\nu)&\\0&&&\end{pmatrix}
\right] \begin{pmatrix} 1 & 0 & \dots & 0\\0 & &&\\\vdots&&b^{n-1}&\\0&&&\end{pmatrix} .
\end{align*}
Note that $\dot E_1(w^1) = \mbox{diag}(\dot w^1)$.  It therefore suffices to show that
$$ \mbox{diag}(\dot w^1) + E_1(w^1) \begin{pmatrix} 0 & 0 & \dots & 0\\0 & &&\\\vdots&&\epsilon^{n-1}(\nu)&\\0&&&\end{pmatrix}
= \epsilon^n(\l) E_1(w^1),$$
which is readily verified using \eqref{nu} and \eqref{vdot}.
\end{proof}

We will now explain how to extend the definition of the flow $S^\eta_t$ to continuous
paths $\eta:[0,\infty)\to\R^n$ with $\eta(0)=0$, by simply solving the equations \eqref{dyn-rsk-pi} in terms
of $\eta$ and then observing that $\eta$ need not be smooth in order for the solution to make sense.

Let $\xi=(\xi^m_i)\in\T$ and $\eta :[0,\infty)\to\R^n$ continuous with $\eta(0)=0$.  
Denote by $\Pi^\xi\eta$ the path $X(t)=(x^m_i(t)),\ t\ge 0$ in $\T$, defined as follows.
Set $\mu_1=\xi^1_1$,
$$\mu_m=\sum_{i=1}^m\xi^m_i-\sum_{i=1}^{m-1}\xi^{m-1}_i,$$
and define
\be\label{pi}
\pi(t)=\eta(t)+\sum_{i=1}^n \mu_i e_i.
\ee
For $2\le k\le m\le n$, define 
\be\label{def-r}
r^m_k  = \sum_{i=1}^{k-1} \xi^{m-1}_i - \sum_{i=1}^{k-1} \xi^{m}_i.
\ee
Set
\be\label{f1}
x^1_1(t)=\pi_1(t),
\ee
and, for $2\le m\le n$,
\be\label{f2}
x^m_m(t)=\pi_m(t)-\log\left[ e^{-r^m_m} + \int_0^t e^{\pi_m(s)-x^{m-1}_{m-1}(s)} ds\right].
\ee
For $2\le m\le n$ and $1<k<m$,
\be\label{f3}
x^m_k(t) = y^m_k(t)
- \log \left[ e^{-r^m_k}+\int_0^t e^{y^m_k(s)-x^{m-1}_{k-1}(s)} ds\right] ,
\ee
where
\be\label{f3a}
y^m_k(t)=x^{m-1}_k(t)+ \log \left[ e^{-r^m_{k+1}}+\int_0^t e^{y^m_{k+1}(s)-x^{m-1}_k(s)} ds\right] ;
\ee
for $2\le m\le n$, writing $y^2_2=\pi_2$,
\be\label{f4}
x^m_1(t)=x^{m-1}_1(t)+ \log \left[ e^{-r^m_{2}}+\int_0^t e^{y^m_{2}(s)-x^{m-1}_1(s)} ds\right].
\ee

The mapping $\Pi^\xi$ can be written more compactly as follows.
For $r\in\R$, $i=1,\ldots,n-1$, and continuous $\eta :[0,\infty)\to\R^n$, define 
$$(P^r_i \eta) (t) = \eta(t) + \left( \log \left[ e^{-r}+\int_0^t e^{\eta_{i+1}(s)-\eta_i(s)} ds\right]\right) (e_i-e_{i+1}).$$
Define $\Pi^\xi_1$ by $\Pi^\xi_1\eta(t)=\pi(t)$, where $\pi$ is defined as above by \eqref{pi}.
For $1< m\le n$, set
$$\Pi^{\xi}_{m} = P^{r^{m}_2}_1\circ\cdots\circ P^{r^{m}_{m}}_{m-1} \circ \Pi^\xi_{m-1}.$$
Then $\Pi^\xi\eta=((\Pi^{\xi}_m\eta)_i,\ 1\le i\le m\le n)$.

\begin{prop}\label{S-extend}
Let $\eta:[0,\infty)\to\R^n$ be a smooth path with $\eta(0)=0$.
Then, for $t\ge 0$,
$$S^\eta_t\xi=\Pi^\xi\eta(t).$$
Also, for $t>0$, $\Pi\eta(t)$ evolves according to \eqref{dyn-rsk-pi}.
\end{prop}
\begin{proof}
Let $X(t)=(x^m_i(t))$ be defined by \eqref{f1}---\eqref{f4}.
For convenience, write $y^m_m(t)=\pi_m(t)$ and observe that, for $2\le k\le m\le n$,
\be\label{key}
\frac{d}{dt}  \log \left[ e^{-r^m_k}+\int_0^t e^{y^m_k(s)-x^{m-1}_{k-1}(s)} ds\right] = e^{x^m_{k}-x^{m-1}_{k-1}}.
\ee
First note that $\dot x^1_1=\dot \eta_1$ and, using \eqref{key},
$$\dot x^m_m = \dot \eta_m - e^{x^m_m-x^{m-1}_{m-1}}$$
for $2\le m\le n$.  For $2\le m\le n$ and $1< k< m$, by \eqref{f3}, \eqref{f3a} and \eqref{key},
$$\dot x^m_k = \dot y^m_k - e^{x^m_k-x^{m-1}_{k-1}}
= \dot x^{m-1}_k + e^{x^m_{k+1}-x^{m-1}_k} - e^{x^m_k-x^{m-1}_{k-1}}.$$
For $2\le m\le n$, by \eqref{f4} and \eqref{key}, 
$$\dot x^m_1 = \dot x^{m-1}_1 + e^{x^m_2-x^{m-1}_1}.$$
We have thus shown that $X(t)$ satisfies \eqref{dyn-rsk-pi}, and
it remains to check that $X(0)=\xi$. 

It follows immediately from the definitions \eqref{pi}, \eqref{f1} and \eqref{f2}
that 
$$x^m_m(0)=\mu_m+r^m_m=\xi^m_m$$ 
for $1\le m\le n$.  
For $2\le m\le n$ and $1< k< m$, from \eqref{f3} and \eqref{f3a},
$$x^m_k(0)=y^m_k(0)+r^m_k=x^{m-1}_k(0)-r^m_{k+1}+r^m_k=x^{m-1}_k(0)+\xi^m_k-\xi^{m-1}_k;$$
for $2\le m\le n$, from \eqref{f4},
$$x^m_1(0)=x^{m-1}_1(0)-r^m_2=x^{m-1}_1(0) +\xi^m_1-\xi^{m-1}_1.$$
It follows that $x^m_i(0)=\xi^m_i$ for $1\le i\le m\le n$, as required.

The second claim also follows from the above argument, taking $e^{-r^m_k}=0$ for $2\le k\le m\le n$.
\end{proof}

Thus, for $\xi\in\T$, $b\in\P$ and continuous $\eta:[0,\infty)\to\R^n$ with $\eta(0)=0$, 
we define, for $t\ge 0$, 
\be\label{def-sr}
S^\eta_t\xi=\Pi^\xi\eta(t),\qquad R^\eta_t b=f^{-1}(S^\eta_t f(b)).
\ee

Let $\eta:[0,\infty)\to\R^n$ be a continuous path with $\eta(0)=0$.
Set $b(t)=R^\eta_t b(0)$ with either $b(0)\in\P$ or $b(0)=I$. 
Then, for each $t>0$, $b(t)\in\P$ and we can define $X(t)=(x^m_i(t))=f(b(t))$.
If $b(0)=I$ then $X(t)=\Pi\eta(t)$.  If $b(0)\in\P$, then $X(t)=S^\eta_t\xi=\Pi^\xi\eta(t)$
where $\xi=f(b(0))$.  By Proposition~\ref{umi}, for each $t>0$, $b(t)\bar w_0$ has a 
Gauss decomposition 
$$b(t)\bar w_0=L(t)D(t)U(t)$$ where
$$D_{ii}=e^{x^n_i},\qquad 1\le i\le n$$
and $L\in (N_-)_{>0}$ is given by
$$L=L_1(u^1) L_2(u^2) \dots L_{n-1}(u^{n-1}) ,$$
$$u^m_i=e^{x^{m+1}_{i+1}-x^{m}_i},\qquad 1\le i\le m<n.$$
For a square matrix $A$ denote by $\Pi_-(A)$ the strictly lower triangular part of $A$.
The next proposition is essentially a special case of \cite[Proposition 6.4]{bbo2}.
\begin{prop}\label{sum}  If $\eta$ is smooth, then $L(t)$ and $R(t)=D(t)U(t)$ satisfy
$$\dot L = L\Pi_{-}(L^{-1}\epsilon(\dot\eta) L),\qquad \dot R=\epsilon(\dot x^n) R.$$
\end{prop}
\begin{proof}  We have $\dot b=\epsilon(\dot\eta) b$, hence
$$\dot L R+L\dot R=\epsilon(\dot\eta) LR,$$
or, equivalently,
$$L^{-1}\dot L+\dot R R^{-1} = L^{-1}\epsilon(\dot\eta) L.$$
Now, since $\dot R R^{-1}$ is upper triangular and $L^{-1}\dot L$ is strictly lower triangular,
this implies 
$$L^{-1} \dot L=\Pi_-(L^{-1} \dot L +\dot R R^{-1}) = \Pi_{-}(L^{-1}\epsilon(\dot\eta) L),$$ proving the first claim.
Similarly, the strictly upper triangular part of $\dot R R^{-1}$ must equal 
the strictly upper triangular part of $L^{-1}\epsilon(\dot\eta) L$, which is just the shift matrix $\epsilon(0)$;
also, since $R=DU$, the diagonal part of $\dot R R^{-1}$ is $\dot D D^{-1}=\mbox{diag}(\dot x^n)$.
Hence, $\dot R R^{-1}=\mbox{diag}(\dot x^n)+\epsilon(0)=\epsilon(\dot x^n)$, as required.
\end{proof}

\section{Whittaker functions and the quantum Toda lattice}\label{s-wfqt}

Following~\cite{g,jk,gklo}, for $X=(x^m_i)\in \T$ and $\l\in\R^n$, we define, for $n\ge 2$,
\be\label{F}
\F(X)=\sum_{1\le i\le m<n} e^{x^{m+1}_{i+1}-x^m_i} + e^{x^m_i-x^{m+1}_i}
\ee
and
\be\label{Flambda}
\F_\l(X)=\sum_{m=1}^n \l_m\left(\sum_{i=1}^{m-1} x^{m-1}_i - \sum_{i=1}^{m} x^{m}_i \right) + \F(X).
\ee
If $n=1$, we set $\F(X)=0$ and $\F_\l(X)=-\l_1 x^1_1$.

For some readers, the following graphical representation may be helpful for understanding the above definition,
and also for following some of the proofs which will be given later.  We view a triangle $X=(x^m_i)$ as an array:
\begin{center}
\begin{tikzpicture}[scale=.7]
\node at (0,0) {$x^n_n$};
\draw [dotted, thick] (0.8,0.8) -- (1.2,1.2);
\node at (2,2) {$x^2_2$};
\node at (3,3) {$x^1_1$};
\node at (4,2) {$x^2_1$};
\draw [dotted, thick] (4.8,1.2) -- (5.2,0.8);
\node at (6,0) {$x^n_1$};
\draw [dotted, thick] (2.7,0) -- (3.3,0);
\node at (-2,1.5) {$X=$};
\end{tikzpicture}
\end{center}
The elements of this array are connected by arrows as shown, with the obvious omissions
at the boundary:

\begin{center}
\begin{tikzpicture}[scale=1]
\node at (2,2) {$x^{m-1}_i$};
\draw [->] [draw=gray, thick] (2.3,1.7)--(2.7,1.3);
\node at (3,1) {$x^m_i$};
\draw [->] [draw=gray, thick] (3.3,.7)--(3.7,.3);
\node at (4,0) {$x^{m+1}_i$};
\draw [->] [draw=gray, thick] (2.3,0.3)--(2.7,.7);
\draw [->] [draw=gray, thick] (3.3,1.3)--(3.7,1.7);
\node at (2,0) {$x^{m+1}_{i+1}$};
\node at (4,2) {$x^{m-1}_{i-1}$};
\end{tikzpicture}
\end{center}
To an arrow $a\to b$, we associated the weight $e^{a-b}$.  Then $\F(X)$ is just the sum
of the weights associated with the arrows in the diagram of $X$.

For $x\in\R^n$, denote by $\T(x)$ the set of triangles $X=(x^m_i)\in\T$ with bottom row $x^n=x$.
Let $\epsilon>0$ and define
$$\psi_\l(x)=\int_{\T(x)} e^{- \F_\lambda(X)/\epsilon} \prod_{1\le i\le m<n} dx^m_i .$$
These are eigenfunctions of the quantum Toda lattice with Hamiltonian
$$H=-\epsilon\Delta+\frac2\epsilon\sum_{i=1}^{n-1} e^{x_{i+1}-x_i},$$
also known as $GL(n,\R)$-Whittaker functions~\cite{k,g,jk,gklo,is}.

In \cite{r} it was shown that, for each $x\in\R^n$, the function $\F(X)$ is strictly convex and has a unique critical point 
on $\T(x)$, which is a minimum.  This property extends trivially to $\F_\l(X)$, as we are simply adding a linear functional.
Denote the corresponding critical point by $X_\l^*(x)$.   Note that, since $\F(X)$ is strictly convex, $X^*_\l(x)$ is a 
continuous function of $\l$.

For $X=(x^m_i)\in \T$ and, for $1\le i\le m<n$, define $l^m_i=l^m_i(X)$
and $r^m_i=r^m_i(X)$ as follows.  For $1\le i<m<n$,
$$l^m_i=e^{x^{m+1}_{i+1}-x^m_i}+e^{x^{m-1}_i-x^m_i},$$
and, for $1\le m<n$,
$$l^m_m=e^{x^{m+1}_{m+1}-x^m_m}.$$
Similarly, for $1<i\le m<n$,
$$r^m_i=e^{x^m_i-x^{m+1}_i}+e^{x^m_i-x^{m-1}_{i-1}},$$
and, for $1\le m<n$,
$$r^m_1=e^{x^m_1-x^{m+1}_1}.$$
For each $x\in\R^n$, the critical point $X=X^*_\l(x)$ of $\F_\l$ on $\T(x)$ satisfies
\be\label{cp}
\l_m+l^m_i(X)=\l_{m+1}+r^m_i(X),\qquad 1\le i\le m<n.
\ee
For $\lambda\in\R^n$, denote by $\T_\lambda$ the set of $X=(x^m_i)\in \T$ which satisfy the 
critical point equations \eqref{cp}.  Note that each element $X=(x^m_i)\in\T_\l$ is uniquely determined
by its bottom row $x^n$, via $X=X^*_\l(x^n)$.  

\begin{rem} The critical point equations \eqref{cp} are closely related to the geometric 
Bender-Knuth transformations introduced in~\cite{ki}.  
These are birational involutions $b^m_i$, $1\le i\le m<n$ 
defined on $\T$ as follows: if $X=(x^m_i)\in\T$ then $b^m_i(X)$ is obtained from $X$ by replacing 
$x^m_i$ with
$$\tilde x^m_i = x^m_i+ \log \frac{l^m_i(X)}{r^m_i(X)} ,$$
and leaving the other entries unchanged.  Thus, $X\in\T_0$ if, and only if, it is invariant under all of 
the geometric Bender-Knuth transformations $b^m_i$, $1\le i\le m<n$.  We note that this implies, 
in particular, that each $X\in\T_0$ is a fixed point of the geometric lifting of Schutzenberger's involution, 
defined in~\cite{ki} as a composition of geometric Bender-Knuth transformations.
\end{rem}

\section{Geometric RSK and Brownian motion}\label{grsk-bm}

In this section we recall some of the main results from~\cite{noc} which relate the continuous-time 
geometric RSK mapping, with Brownian motion as input, to the quantum Toda lattice.
\begin{thm}\cite{noc}\label{noc1}
If $\eta(t)$ is a Brownian motion in $\R^n$ with $\eta(0)=0$, infinitesimal variance 
$\epsilon$ and drift $\lambda\in\R^n$, then $x(t)=\Pi_n\eta(t)$ is a diffusion process in $\R^n$ with infinitesimal
generator given by
$$\L_\lambda = -\frac12 \psi_\l(x)^{-1} \left( H+\frac1\epsilon\sum_i\l_i^2\right) \psi_\l(x)
= \frac\epsilon2 \Delta + \epsilon \nabla\log \psi_\l \cdot \nabla .
$$
\end{thm}
We note that in~\cite{noc} this statement was proved in the case $\epsilon=1$, but this is minor
modification.  In the case $n=2$, it is equivalent to a Theorem of Matsumoto and Yor~\cite{my}.
The diffusion process with generator $\L_\l$ was first introduced in~\cite{boc}, in the context of
more general root systems.

In the paper~\cite{noc} more general initial conditions were also considered, and two quite
different stochastic dynamics on triangles, namely:
\begin{eqnarray}\label{dyn-rsk-pi-sde}
& dx^1_1=dB_1+\l_1dt; \\ 
& dx^m_m=dB_m+(\l_m-e^{x^m_m-x^{m-1}_{m-1}})dt,\qquad 2\le m\le n;\nn \\ 
& d x^m_1=d x^{m-1}_1 + e^{x^m_2-x^{m-1}_1}dt,\qquad 2\le m\le n;\nn\\ 
\nn & d x^m_i = d x^{m-1}_i + (e^{x^m_{i+1}-x^{m-1}_i} - e^{x^m_i-x^{m-1}_{i-1}})dt,\qquad 1< i< m\le n
\end{eqnarray}
and
\begin{eqnarray}\label{dyn-sde}
& d x^1_1=dW^1_1+\l_1dt;  \\ 
& d x^m_m=dW^m_m+(\l_m-e^{x^m_m-x^{m-1}_{m-1}})dt,\qquad 2\le m\le n;\nn \\ 
& d x^m_1=dW^m_1+(\l_m + e^{x^{m-1}_1-x^m_1})dt,\qquad 2\le m\le n;\nn \\ \nn
& d x^m_i = dW^m_i+(\l_m + e^{x^{m-1}_i-x^m_i} - e^{x^m_i-x^{m-1}_{i-1}})dt,\qquad 1< i< m\le n
\end{eqnarray}
where $B_i,\ W^m_i,\ 1\le i\le m\le n$ are independent one-dimensional Brownian motions
(without drift) and, for the purposes of this discussion, each with infinitesimal variance $\epsilon$.

The first of these, \eqref{dyn-rsk-pi-sde}, describes the (stochastic) evolution of $\Pi\eta(t)$,
or $\Pi^\xi\eta(t)$ any initial condition $\xi\in\T$, where $\eta(t)=B(t)+t\l$,
as can be seen, for example, from the proof of Proposition~\ref{S-extend}
above.  The second dynamic \eqref{dyn-sde} is a geometric lifting of Warren's process~\cite{w}.

\begin{thm} \cite{noc}\label{noc2}
For each $x\in\R^n$, if the initial condition $X(0)$ is chosen at random according to the 
probability measure on $\T(x)$ with density proportional to $e^{- \F_\lambda(X)/\epsilon}$ and $X(t)$ evolves 
according to \eqref{dyn-rsk-pi-sde} or \eqref{dyn-sde}, then $x^n(t),\ t\ge 0$ is a diffusion process with infinitesimal 
generator $\L_\l$, started at $x$.  Moreover, for each $t>0$, the conditional law of $X(t)$ given $x^n(s),\ 0\le s\le t$
is supported on $\T(x^n(t))$ with density proportional to $e^{- \F_\lambda(X)/\epsilon}$.
\end{thm}

This theorem can be represented as a commutative diagram, as follows.
Denote by $Q^\l_t$ the Markov semigroup associated with the diffusion with infinitesimal generator $\L_\l$,
that is, $Q^\l_t=e^{t\L_\l}$, by $P^\l_t$ the Markov semigroup associated with the Markov process defined
by {\em either} \eqref{dyn-rsk-pi-sde} or \eqref{dyn-sde}, and define Markov kernels $\Sigma_\l$ from $\R^n$
to $\T$ and $\pi$ from $\T$ to $\R^n$, by
$$(\Sigma_\l f)(x)=\psi_\l(x)^{-1} \int_{\T(x)} e^{- \F_\lambda(X)/\epsilon} f(X) \prod_{1\le i\le m< n} dx^m_i$$
for suitable $f:\T\to\R$ and, writing $X=(x^m_i)\in\T$,
$$(\pi g)(X) = g(x^n).$$
Then, according to Theorem~\ref{noc2}, the following diagram commutes:

\begin{center}
\begin{tikzpicture}
\node at (2.7,2) {$\T$};
\draw [->] (4.5,2)--(3.5,2);
\node at (4,2.3) {$\Sigma_\l$};
\node at (5.3,2) {$\R^n$};
\node at (2.7,0) {$\T$};
\node at (4,.3) {$\pi$};
\node at (5.3,0) {$\R^n$};
\draw [->] (3.5,0)--(4.5,0);
\draw [->] (2.7,1.5)--(2.7,0.5);
\draw [->] (5.3,1.5)--(5.3,0.5);
\node at (2.3,1) {$P^\l_t$};
\node at (5.7,1) {$Q^\l_t$};
\end{tikzpicture}
\end{center}

Theorem~\ref{noc1} is a generalisation of Pitman's `$2M-X$' theorem~\cite{pitman}, which states
that, if $X_t$ is a standard one-dimensional Brownian motion and $M_t=\max_{0\le s\le t} X_s$,
then $2M-X$ is a three-dimensional Bessel process.  Pitman's theorem was generalised to the
type $A_{n-1}$ case (from $A_1$) in~\cite{bj,oy} and to arbitrary finite Coxeter groups in~\cite{bbo1,bbo2}.  
For discrete versions, see~\cite{oc03a,oc03b,llp}.  These generalisations are closely related to the 
RSK correspondence, and also longest increasing subsequences, percolation and queues~\cite{bdj,b,gtw}.
Pitman's theorem was extended to the geometric setting by Matsumoto and Yor~\cite{my}, and Theorem~\ref{noc1} 
can be regarded as a geometric lifting of the generalisations of Pitman's theorem for type $A_{n-1}$ given 
in~\cite{bj,oy}.  It has been generalised to other types in~\cite{ch}.  A discrete-time version is given in~\cite{cosz}
(see also~\cite{osz}) in the context of Kirillov's geometric RSK mapping on matrices, and a fully
discrete $q$-version, in the context of Ruijenaars $q$-Toda difference operators and $q$-Whittaker 
functions, in~\cite{op} (see also~\cite{bp,pei}).

The second dynamics \eqref{dyn-sde} is a geometric lifting of a process on Gelfand-Tsetlin patterns 
introduced by Warren~\cite{w}.  
Nordenstam~\cite{n} showed that a discrete version of Warren's process is in fact closely related to a 
{\em shuffling algorithm} which was previously studied in the random tilings literature~\cite{eklp,p}, 
see also~\cite{bf,ww,bc}.  
For some time now it has been a natural question to 
understand the relationship between these two types of dynamics.  As we will see, there is one setting
in which the answer is simple: in the semi-classical ($\epsilon\to 0$) limit of Theorem~\ref{noc2}, they are, in fact, equivalent!

\section{A semi-classical limit}\label{scl}

Theorem~\ref{noc1} can be restated as follows:  if $\eta(t)$ is a Brownian motion in $\R^n$ with 
$\eta(0)=0$, infinitesimal variance $\epsilon$ and drift $\lambda\in\R^n$, then $x(t)=\Pi_n\eta(t)$ is 
a weak solution to the stochastic differential equation
\be\label{sde}
dx = \sqrt{\epsilon} dW + \epsilon  \nabla\log \psi_\l (x) dt
\ee
where $W$ is a standard Brownian motion in $\R^n$.
From the definition of $\psi_\l$, formally taking the limit in \eqref{sde} as $\epsilon\to0$  yields
\be\label{gradu}
\dot x = - \nabla u_\l(x),
\ee
where, in the notation of Section \ref{s-wfqt},
\be\label{u}
u_\l(x)=\F_\l(X_\l^*(x)).
\ee
In the case $\l=0$, this gradient flow is discussed by Givental~\cite{g}, where it is shown to be equivalent to
the Toda flow (with opposite sign) on its most degenerate iso-spectral manifold on which the eigenvalues of
the Lax matrix are all equal to zero.  In fact, the corresponding statement holds true for any $\l\in\R^n$, namely,
that the gradient flow \eqref{gradu} describes the Toda flow on the iso-spectral manifold corresponding to $\l$.
For more details, see Section~\ref{sec-toda} and Theorem~\ref{main1} below.

On the other hand, when $\epsilon= 0$, $\eta(t)=t\l$ almost surely.  This suggests that the image of the path 
$\eta(t)=t\l$ under $\Pi_n$ defines a solution to the Toda flow on the iso-spectral manifold corresponding to $\l_1,\ldots,\l_n$, 
and indeed this is the case, as we will show in Theorem~\ref{main1}.

For more general initial conditions, in the context of Theorem~\ref{noc2},
note that when $\epsilon=0$, \eqref{dyn-rsk-pi-sde} becomes 
\begin{eqnarray}\label{dyn-rsk1}
& \dot x^1_1=\l_1; \\ 
& \dot x^m_1=\dot x^{m-1}_1 + e^{x^m_2-x^{m-1}_1},\qquad \dot x^m_m=\l_m-e^{x^m_m-x^{m-1}_{m-1}},\qquad 2\le m\le n;\nn \\ \nn
& \dot x^m_i = \dot x^{m-1}_i + e^{x^m_{i+1}-x^{m-1}_i} - e^{x^m_i-x^{m-1}_{i-1}},\qquad 1< i< m\le n
\end{eqnarray}
and \eqref{dyn-sde} becomes
\begin{eqnarray}\label{dyn1}
& \dot x^1_1=\l_1;  \\ 
& \dot x^m_1=\l_m + e^{x^{m-1}_1-x^m_1},\qquad \dot x^m_m=\l_m-e^{x^m_m-x^{m-1}_{m-1}},\qquad 2\le m\le n;\nn \\ \nn
& \dot x^m_i = \l_m + e^{x^{m-1}_i-x^m_i} - e^{x^m_i-x^{m-1}_{i-1}},\qquad 1< i< m\le n.
\end{eqnarray}
Moreover, the fixed time marginals of the process $X(t)$, which are the same in either case, 
are concentrated on $\T_\l$.
As we will see, $\T_\l$ is stable under both of the flows defined by \eqref{dyn-rsk1} and \eqref{dyn1},
and on $\T_\l$, they are in fact equivalent.  Moreover, we will show that if $X(0)\in\T_\l$ and $X(t)$
evolves according to either/both, then $x^n(t)$ defines a solution to the Toda flow with opposite sign
on the iso-spectral manifold corresponding to $\l$.  This last statement can be interpreted as a semi-classical limit
of Theorem \ref{noc2}.  From this we will also deduce the semi-classical limit of Theorem \ref{noc1}.  
For precise statements, see Theorems \ref{main} and \ref{main1} below.

\section{The Toda lattice}\label{sec-toda}

The Toda lattice is completely integrable Hamiltonian system which has been
extensively studied in the literature, see for example the survey~\cite{ks}.
We will consider the Toda lattice with opposite sign, with Hamiltonian
\be
\frac12\sum_{i=1}^n p_i^2 - \sum_{i=1}^{n-1} e^{x_{i+1}-x_i}.
\ee
This is a special case of the {\em indefinite} Toda lattice~\cite{ks,ky}.
The equations of motion are
\begin{eqnarray}\label{toda1}
&\ddot x_1=-e^{x_2-x_1}\qquad \ddot x_n=e^{x_n-x_{n-1}};\\ \nn
&\ddot x_i=-e^{x_{i+1}-x_i}+e^{x_i-x_{i-1}},\qquad \qquad i=2,\ldots,n-1.
\end{eqnarray}
Writing $q_i=e^{x_{i+1}-x_i}$ for $1\le i\le n-1$, set
\be\label{M}
M=\begin{pmatrix} p_1 & 1 & 0 &  \cdots & 0\\
-q_1 & p_2 & 1 &  \cdots & 0\\
0 & -q_2 & \ddots &  \ddots&  \vdots\\
&&&&0\\
\vdots&&&\ddots&1\\
0&\dots&0&-q_{n-1}&p_n \end{pmatrix},
\ee
and 
\be
Q=\Pi_-(M)=\begin{pmatrix} 0 & 0    &\dots& 0\\
-q_1 & 0   &\dots& 0\\
\vdots   & \ddots &  & \vdots\\
&&0&0\\
0&\dots&-q_{n-1}&0 \end{pmatrix}.
\ee
Then $(M,Q)$ form a Lax pair, that is, 
\be\label{lax}
\dot{M} = [M,Q]
\ee
and \eqref{lax} is equivalent to the equations of motion \eqref{toda1} or, equivalently,
\begin{eqnarray}\label{em}
&\dot q_i = (p_{i+1}-p_i) q_i,\qquad i=1,\ldots,n-1;\\ \nn
&\dot p_1 = -q_1;\qquad \dot p_n = q_{n-1};\\ \nn
&\dot p_i = q_{i+1}-q_i,\qquad i=2,\ldots,n-1.
\end{eqnarray}
In particular, the eigenvalues $\lambda_1,\ldots,\lambda_n$ of $M$ form a complete set of integrals of 
motion for the system.  Denote by $\M$ the set of complex, tridiagonal, Hessenberg matrices and by
$\M_\l$ the subset of those matrices with eigenvalues given by $\l=(\lambda_1,\ldots,\lambda_n)$.

The relation to the usual Toda lattice with Hamiltonian 
$$\frac12\sum_{i=1}^n \pi_i^2 + \sum_{i=1}^{n-1} e^{\xi_{i+1}-\xi_i}$$
is as follows.  If $x(t)$ is a solution to the opposite sign Toda lattice which can be
analytically continued in the time variable $t$ then, formally at least, 
$\xi(t)=x(\iota t)$ defines a solution to the usual Toda lattice, that is
\begin{eqnarray}\label{toda}
&\ddot \xi_1=e^{\xi_2-\xi_1}\qquad \ddot \xi_n=-e^{\xi_n-\xi_{n-1}};\\ \nn
&\ddot \xi_i=e^{\xi_{i+1}-\xi_i}-e^{\xi_i-\xi_{i-1}},\qquad \qquad i=2,\ldots,n-1.
\end{eqnarray}

To solve \eqref{lax} for a given initial condition $M_0\in\M$ we write, for each $t\ge 0$, 
$$e^{tM_0}=n(t)r(t),$$ where $n(t)\in N_-$ and $r(t)\in B$, assuming for the moment that 
such a factorisation is possible.  Then
$$M(t)=n(t)^{-1}M_0 n(t)=r(t)M_0 r(t)^{-1}$$
defines a solution to \eqref{lax} with $M(0)=M_0$.  The matrices $n(t)$
and $r(t)$ evolve according to
$$\dot n = n Q,\qquad \dot r=P r,$$
where $P=M- Q$.  Denote the corresponding flow on $\M$
by $T_t$, so that $$M(t)=T_tM_0.$$  There may exist times at which the solution blows up,
as discussed for example in~\cite{efs,ks,ky}.

Finally we recall an important result of Kostant~\cite{k1}, 
namely that for any $M\in \M_\l$, there is a unique $L\in N_-$ such that
$$M=L^{-1}\epsilon_\l L .$$
Thus, if $M_0\in \M_\l$ with $M_0=L_0^{-1}\epsilon_\l L_0$, then we can write
$$M(t)=L(t)^{-1}\epsilon_\l L(t)$$
where $L(t)=L_0 n(t)$ evolves according to $\dot L = L Q$ with $L(0)=L_0$.

\section{Flows on triangles and upper triangular matrices}\label{flows}

We consider two flows on $\T$, which we denote by $S^\l_t$ and $\tilde S^\l_t$, and define by
\begin{eqnarray}\label{dyn-rsk}
& \dot x^1_1=\l_1; \\ 
& \dot x^m_1=\dot x^{m-1}_1 + e^{x^m_2-x^{m-1}_1},\qquad \dot x^m_m=\l_m-e^{x^m_m-x^{m-1}_{m-1}},\qquad 2\le m\le n;\nn \\ \nn
& \dot x^m_i = \dot x^{m-1}_i + e^{x^m_{i+1}-x^{m-1}_i} - e^{x^m_i-x^{m-1}_{i-1}},\qquad 1< i< m\le n
\end{eqnarray}
and
\begin{eqnarray}\label{dyn}
& \dot x^1_1=\l_1;  \\ 
& \dot x^m_1=\l_m + e^{x^{m-1}_1-x^m_1},\qquad \dot x^m_m=\l_m-e^{x^m_m-x^{m-1}_{m-1}},\qquad 2\le m\le n;\nn \\ \nn
& \dot x^m_i = \l_m + e^{x^{m-1}_i-x^m_i} - e^{x^m_i-x^{m-1}_{i-1}},\qquad 1< i< m\le n
\end{eqnarray}
respectively.
\begin{prop}\label{stable}
For each $\l\in\R^n$, $\T_\lambda$ is invariant under the flows $S^\l_t$ and $\tilde S^\l_t$ and,
moreover, on $\T_\l$ these flows are equivalent.
\end{prop}
\begin{proof}
We will first show that $\T_\l$ is stable under the dynamics \eqref{dyn-rsk}.  By \eqref{cp},
this is equivalent to showing that $\T_\l$ is stable under the dynamics \eqref{dyn}.
Suppose $X\in\T_\l$.  Then it follows from \eqref{dyn}, using \eqref{cp}, that
\be\label{dyn2}
\dot x^m_i = \l_{m+1} + e^{x^m_i-x^{m+1}_i} - e^{x^{m+1}_{i+1}-x^m_i},\qquad 1\le i\le m<n.
\ee
Using \eqref{dyn} and \eqref{dyn2} we have, for $1<i<m<n$,
\begin{eqnarray*}
\dot l^m_i &=& (\dot x^{m-1}_i-\dot x^m_i) e^{x^{m-1}_i-x^m_i} + (\dot x^{m+1}_{i+1}-\dot x^m_i) e^{x^{m+1}_{i+1}-x^m_i}\\
&=& \left( e^{x^m_i-x^{m-1}_{i-1}} - e^{x^m_{i+1}-x^{m-1}_i} \right) e^{x^{m-1}_i-x^m_i} 
+ \left( e^{x^m_{i+1}-x^{m+1}_{i+1}} - e^{x^m_i-x^{m+1}_i} \right)  e^{x^{m+1}_{i+1}-x^m_i}\\
&=& e^{x^{m-1}_i-x^{m-1}_{i-1}} - e^{x^{m+1}_{i+1}-x^{m+1}_i} ,
\end{eqnarray*}
and
\begin{eqnarray*}
\dot r^m_i &=& (\dot x^{m}_i-\dot x^{m-1}_{i-1}) e^{x^m_i-x^{m-1}_{i-1}} 
+ (\dot x^m_i-\dot x^{m+1}_{i}) e^{x^m_i-x^{m+1}_{i}}\\
&=& \left( e^{x^{m-1}_i-x^{m}_{i}} - e^{x^{m-1}_{i-1}-x^{m}_{i-1}} \right) e^{x^m_i-x^{m-1}_{i-1}} 
+ \left( e^{x^{m+1}_{i}-x^{m}_{i-1}} - e^{x^{m+1}_{i+1}-x^{m}_i} \right)  e^{x^m_i-x^{m+1}_{i}}\\
&=& e^{x^{m-1}_i-x^{m-1}_{i-1}} - e^{x^{m+1}_{i+1}-x^{m+1}_i} .
\end{eqnarray*}
Similarly, for $1\le m<n$, we obtain
$$\dot l^m_1= \dot r^m_1 = - e^{x^{m+1}_{2}-x^{m+1}_1},\qquad \dot l^m_m = \dot r^m_m = - e^{x^{m+1}_{m+1}-x^{m+1}_m}.$$
We conclude that $\dot l^m_i=\dot r^m_i$ for all $1\le i\le m<n$, which shows that $\T_\l$ is stable, as required.
\end{proof}

The flow $S^\l_t$ has a convenient representation in terms of the simple linear flow on $B$ defined by
$$R^\l_t b = e^{t\epsilon_\l}b.$$
\begin{prop}
On $\T$, $S^\l_t=f\circ R^\l_t\circ f^{-1}$.
\end{prop}
\begin{proof}
This follows from Proposition \ref{sr}, taking $\eta(t)=t\l$.
\end{proof}

We will now explain how the restriction of the flow $S^\l_t$ to $\T_\l$ is related to the Toda flow on $\M_\l$.

Define a map $h:\T\to (N_-)_{>0}$, $X=(x^m_i)\mapsto L$, by
$$L=L_1(u^1) L_2(u^2) \dots L_{n-1}(u^{n-1}) ,$$
where
$$u^m_i=e^{x^{m+1}_{i+1}-x^{m}_i},\qquad 1\le i\le m<n.$$
Define another map $g_\l:\T\to\M$, $X=(x^m_i)\mapsto M$, where $M$ is defined by \eqref{M} with
$$q_i=e^{x^n_{i+1}-x^n_i},\qquad 1\le i\le n-1,$$
and $p=p^n$, where $p^m_i$, $1\le i\le m\le n$ are defined by
$$p^m_1=p^{m-1}_1+e^{x^m_2-x^{m-1}_1};\qquad p^m_m=\l_m -e^{x^n_n-x^{n-1}_{n-1}};$$
$$p^m_i=p^{m-1}_i+e^{x^m_{i+1}-x^{m-1}_i}-e^{x^m_i-x^{m-1}_{i-1}},\qquad 1<i<m.$$
Note that, if $X\in\T_\l$, then by~\eqref{cp}, we can write
$$p_1=\l_n+e^{x^{n-1}_1-x^n_1};\qquad p_n=\l_n -e^{x^n_n-x^{n-1}_{n-1}};$$
$$p_i=\l_n+e^{x^{n-1}_i-x^n_i}-e^{x^n_i-x^{n-1}_{i-1}},\qquad 1<i<n.$$

\begin{prop}\label{L}  
Let $X=(x^m_i)\in\T$, $M=g_\l(X)$ and $L=h(X)$.  If $X\in\T_\l$, then 
\be\label{crux}
M=L^{-1}\epsilon_\l L .
\ee 
\end{prop}
\begin{proof}
For $u\in\C^m$, $1\le m< n$, define
$$K_m(u)=\begin{pmatrix} \delta^m(u) & 0\\ 0& I_{n-m-1}\end{pmatrix},$$
where
$$\delta^m(u)=\begin{pmatrix} 1 & 0 & 0& \dots & 0\\
u_1 & 1 & 0 & \dots & 0\\
0 & u_2 & 1 & \dots & 0\\
&&\ddots &&\\
0&\dots&0& u_m & 1\end{pmatrix}.$$
Note that
$$L_m(u)^{-1}=K_m(-u).$$

Let $X\in\T$ and define
$$u^m_i=e^{x^{m+1}_{i+1}-x^{m}_i},\qquad 1\le i\le m<n;$$
$$v^m_i=e^{x^m_i-x^{m+1}_i}, \qquad 1\le i\le m<n;$$
$$q^m_i=e^{x^m_{i+1}-x^m_i},\qquad 1\le i< m\le n;$$
$$p^m_1=p^{m-1}_1+e^{x^m_2-x^{m-1}_1};\qquad p^m_m=\l_m -e^{x^n_n-x^{n-1}_{n-1}};$$
$$p^m_i=p^{m-1}_i+e^{x^m_{i+1}-x^{m-1}_i}-e^{x^m_i-x^{m-1}_{i-1}},\qquad 1<i<m.$$
For $2\le m\le n$, 
write $\l^m=(\l_1,\ldots,\l_m)$, $L_m(u)=L^n_m(u)$, $K_m(u)=K^n_m(u)$, 
$$L^{(m)}=L^m(u^1)\dots L^m_{m-1}(u^{m-1}),$$ and  
$$M^{(m)} =\begin{pmatrix} p^m_1 & 1 & 0 &  \cdots & 0\\
-q^m_1 & p^m_2 & 1 &  \cdots & 0\\
0 & -q^m_2 & \ddots &  \ddots&  \vdots\\
&&&&0\\
\vdots&&&\ddots&1\\
0&\dots&0&-q^m_{m-1}&p^m_m \end{pmatrix}.$$
For $v\in\C^m$ write
$$\epsilon^m(v)=\begin{pmatrix} v_1 & 1  &0  &\dots& 0\\
0 & v_2 & 1  &\dots& 0\\
\vdots &  & \ddots &  & \vdots\\
&&&v_{m-1} &1\\
0&&\dots&&v_m \end{pmatrix}.$$

First we will show that $X\in\T_\l$ implies \eqref{crux}.
We prove this by induction.  

For $m=2$, write $q=q^2_1$, $p_i=p^2_i$, $u=u^1_1$, $v=v^1_1$. 
Note that $q=uv$, $p_1=\l_1+u$, $p_2=\l_2-u$ and, by~\eqref{cp}, $\l_1+u=\l_2+v$.
Then
\begin{eqnarray*}
(L^{(2)})^{-1}\epsilon^2(\l) L^{(2)} &=&
\begin{pmatrix} 1 & 0\\ -u & 1 \end{pmatrix}
\begin{pmatrix} \l_1 & 1\\ 0 & \l_2 \end{pmatrix}
\begin{pmatrix} 1 & 0\\ u & 1 \end{pmatrix}\\
&=& \begin{pmatrix} \l_1+u & 1\\ -u\l_1+\l_2(\l_2-u) & \l_2-u \end{pmatrix}\\
&=& \begin{pmatrix} \l_1+u & 1\\ -q & \l_2-u \end{pmatrix}= M^{(2)},
\end{eqnarray*}
as required.

Now fix $2\le m<n$.  
For $A\in\C^{m\times m}$, write
$$H_m(A)=\begin{pmatrix} &&&& 0\\
&A&&& \vdots\\
&&&& 0\\
&&&&1\\
0&&\dots&&0 \end{pmatrix}.$$
Note that 
$$L_m^{m+1}(u^m)^{-1}=K^{m+1}_m(-u^m) = \delta^m(-u^m)$$ and 
$$u^m_i v^m_i=q^{m+1}_i,\qquad 1\le i<m.$$
Also, since $X\in\T_\l$, \eqref{cp} implies
$$p^{m+1}_1=\l_{m+1}+v^{m}_1;\qquad p^{m+1}_{m+1}=\l_{m+1} -u^{m}_{m};$$
$$p^{m+1}_i=\l_{m+1}+v^{m}_i-u^{m}_{i-1},\qquad 1<i<m+1.$$
Hence
$$M^{(m+1)}-\l_{m+1}I_{m+1}=L^{m+1}_m(u^m)^{-1} H_m(\epsilon^m(v^m)),$$
and so
\be\label{xx}
L^{m+1}_m(u^m) [ M^{(m+1)}-\l_{m+1}I_{m+1} ] L^{m+1}_m(u^m)^{-1} = H_m(\epsilon^m(v^m)) L^{m+1}_m(u^m)^{-1}.
\ee
Now, by \eqref{cp}, 
$$p^m_i=\l_{m+1}+v^m_i-u^m_i,\qquad 1\le i\le m.$$  
Using this, and 
$$u^m_i v^m_{i+1} =q^m_i,\qquad 1\le i<m,$$
the right hand side of \eqref{xx} becomes
$$H_m(\epsilon^m(v^m)) L^{m+1}_m(u^m)^{-1} = H_m(M^{(m)}-\l_{m+1}I_{m}).$$
By the induction hypothesis,
$$M^{(m)}=(L^{(m)})^{-1} \epsilon^m(\l^m) L^{(m)}.$$
Hence, using
$$H_m( \epsilon^m(\l^m) -\l_{m+1}I_{m} ) = \epsilon^{m+1}(\l^{m+1}) - \l_{m+1}I_{m+1}$$
and
$$H_m((L^{(m)})^{-1} A L^{(m)} ) = [L^{m+1}_1(u^1)\dots L^{m+1}_{m-1}(u^{m-1})]^{-1} 
H_m(A) L^{m+1}_1(u^1)\dots L^{m+1}_{m-1}(u^{m-1}), $$
we conclude that
$$M^{(m+1)} = (L^{(m+1)})^{-1} \epsilon^{m+1}(\l^{m+1}) L^{(m+1)},$$
as required.

Now we will show that \eqref{crux} implies $X\in\T_\l$, again by induction.  Write 
Write $\T=\T^n$, $\T_\l=\T_\l^n$ to emphasize their dependence on $n$.  
For $1\le m\le n$, write $X^m=(x^1,\ldots,x^m)$.  

\end{proof}

\begin{rem} It can be shown that, in fact, $X\in\T_\l$ if, and only if, 
$$M^{(m)}=(L^{(m)})^{-1} \epsilon^{m}(\l^{m}) L^{(m)},$$
for each $2\le m\le n$.
\end{rem}

In the next Theorem, in the case $\l=0$, the formula \eqref{giv1} and the gradient flow 
representation \eqref{giv2} are due to Givental~\cite{g}.

\begin{thm}\label{main}
Let $x,\l\in\R^n$ and define $X(t) = (x^m_i(t)) = S^\l_t X(0)$ where $X(0)\in\T_\l$.  Let $M(t)=g_\l(X(t))$,
$b(t)=f^{-1}(X(t))$, and $$L=L_1(u^1) L_2(u^2) \dots L_{n-1}(u^{n-1}) ,$$
where
$$u^m_i=e^{x^{m+1}_{i+1}-x^{m}_i},\qquad 1\le i\le m<n.$$
Let $Q=\Pi_-(M)$ and $P=M-Q$.
Then:
\begin{itemize}
\item[\rm (i)]   For all $t\ge 0$, $$M(t)=L(t)^{-1}\epsilon_\l L(t)$$ 
and we have the Gauss decompositions
$$b(t) \bar w_0 = L(t) R(t),\qquad e^{t M(0)}=n(t) r(t)$$
where $n(t)=L(0)^{-1} L(t)$, $R(t)\in B$, $r(t)=R(t) R(0)^{-1}$ and these satisfy
$$\dot L = L Q,\qquad \dot R=P R,\qquad \dot n = n Q, \qquad \dot r=P r.$$
In particular, $M(t)$ defines a solution to the Toda flow on $\M_\l$. 
\item[\rm (ii)] The eigenvalues $\l_1,\ldots,\l_n$ are given by
\be
\l_1=\dot x^1_1,\qquad \l_m=\sum_{i=1}^m \dot x^m_i - \sum_{i=1}^{m-1} \dot x^{m-1}_i, \quad 2\le m\le n.
\ee
\item[\rm (iii)] Writing $x=x(t)=x^n(t)$,
\be\label{giv1}
\F(X)=(n-1)\dot x_1+(n-3) \dot x_2 + \cdots + (1-n) \dot x_n
\ee
and
\be\label{giv2}
\dot x = -\nabla_x u_\l(x).
\ee
\end{itemize}
\end{thm}
\begin{proof}

(i) First note that, since $b(t)\in\P$ for all $t\ge 0$, by Proposition~\ref{umi},
we have the Gauss decomposition $b(t) \bar w_0=L(t)R(t)$ where $L(t)$ is defined as in the statement of the Theorem.
It follows from Propositions~\ref{stable} and \ref{L} that $M(t)=L(t)^{-1}\epsilon_\l L(t)$ for all $t\ge 0$.  
Thus, by Proposition~\ref{sum}, $\dot L=LQ$, which implies the Lax equation $\dot M=[M,Q]$.
We also have from Proposition~\ref{sum} that $\dot R=PR$.  Applying Proposition~\ref{L} at $t=0$,
we have $M(0)=L(0)^{-1}\epsilon_\l L(0)$; it follows, using $$L(t)R(t)=e^{t\epsilon_\l} L(0)R(0),$$ that
$$e^{t M(0)}=L(0)^{-1}e^{t\epsilon_\l} L(0)=L(0)^{-1}L(t) R(t) R(0)^{-1}.$$
Moreover, defining $n(t)=L(0)^{-1} L(t)$ and $r(t)=R(t) R(0)^{-1}$, we have $\dot n = n Q$ and $\dot r=P r$,
as required.

Here is another, simple direct proof that $x=x^n$ satisfies the Toda equations \eqref{toda1}.
For convenience, write $y=x^{n-1}$.  First suppose $1<i<n$.  By \eqref{dyn},
$$\dot x_i = \l_n + e^{y_i-x_i} - e^{x_i-y_{i-1}}$$
and, by \eqref{dyn2},
$$\dot y_i = \l_n + e^{y_i-x_i} - e^{x_{i+1}-y_i},\qquad \dot y_{i-1} = \l_n + e^{y_{i-1}-x_{i-1}} - e^{x_i-y_{i-1}}.$$
Hence
\begin{eqnarray*}
\ddot x_i &=& (\dot y_i-\dot x_i) e^{y_i-x_i} + (\dot y_{i-1}-\dot x_i) e^{x_i-y_{i-1}} \\
&=& \left( e^{x_i-y_{i-1}} - e^{x_{i+1}-y_i}\right) e^{y_i-x_i} + \left( e^{y_{i-1}-x_{i-1}} - e^{y_i-x_i}\right) e^{x_i-y_{i-1}} \\
&=& -e^{x_{i+1}-x_i}+e^{x_i-x_{i-1}}.
\end{eqnarray*}
For $i=1$, we have 
$$\dot x_1=\l_n+e^{y_1-x_1},\qquad \dot y_1=\l_n+e^{y_1-x_1}-e^{x_2-y_1},$$ 
and hence
$$\ddot x_1 = (\dot y_1-\dot x_1) e^{y_1-x_1} =-e^{x_2-x_1};$$
for $i=n$, 
$$\dot x_n=\l_n-e^{x_n-y_{n-1}},\qquad \dot y_{n-1}=\l_n-e^{x_n-y_{n-1}}+e^{y_{n-1}-x_{n-1}},$$ 
and we obtain
$$\ddot x_n = (\dot y_{n-1}-\dot x_n) e^{x_n-y_{n-1}}  = e^{x_n-x_{n-1}},$$
as required.  

Part (ii) follows immediately from \eqref{dyn-rsk}.

(iii) We will prove \eqref{giv1} and \eqref{giv2} by induction. 
Write $Y_\l^n$, $\T=\T^n$, $\T_\l=\T_\l^n$, $\F_\l^n$ and $u_\l^n$ to emphasize 
their dependence on $n$.  
For $1\le m\le n$, write $\l^m=(\l_1,\ldots,\l_m)$ and $X^m=(x^1,\ldots,x^m)\in\T_{\l^m}$.  
We will show that, for each $1\le m\le n$,
\be\label{giv1a}
\F^m(X^m)=(m-1)\dot x^m_1+(m-3) \dot x^m_2 + \cdots + (1-m) \dot x^m_m
\ee
and
\be\label{giv2a}
\dot x^m = -\nabla_{x^m} u^m_{\l^m}(x^m).
\ee
First we show \eqref{giv1a}.  
For $m=1$, $\F(X)=0$ and the result holds trivially. 
Let $m>1$ and write $x=x^m$ and $y=x^{m-1}$.  
$$\F^2(X^2)=e^{x_2-y_1}+e^{y-x_1}.$$
By \eqref{dyn}, $\dot x_1=\l_2+e^{y-x_1}$ and $\dot x_2=\l_2-e^{x_2-y}$, and so
$\F^2(X^2)=\dot x_1-\dot x_2$, as required.
Now suppose $m> 2$.  Note that 
$$\F^m(X^m)=\F^{m-1}(X^{m-1})+\E^m,$$ where
$$\E^m=e^{y_1-x_1}+e^{x_2-y_1}+\cdots+e^{x_m-y_{m-1}}.$$
By the induction hypothesis,
$$\F^{m-1}(X^{m-1})=(m-2) \dot y_1+(m-4)\dot y_2 \cdots (2-m)\dot y_{m-1}.$$
Adding $\E_m$ and using \eqref{dyn2}, then \eqref{dyn}, gives
\begin{eqnarray*}
\F^m(X^m) &=& (m-1) e^{y_1-x_1} + (m-3) (e^{y_2-x_2}-e^{x_2-y_1})+\cdots+(1-m)(-e^{x_m-y_{m-1}})\\
&=& (m-1)\dot x_1+(m-3) \dot x_2 + \cdots + (1-m) \dot x_m,
\end{eqnarray*}
as required.  

Now we will prove \eqref{giv2a}.  For $m=1$, $u_\l(x)=-\l_1 x^1_1$ and the result holds trivially. 
Let $m>1$ and write $x=x^m$ and $y=x^{m-1}$.  Set
$$\E^m_a=\E^m+a\left(\sum_{j=1}^{m-1}y_j-\sum_{i=1}^m x_i\right),$$
and note that
$$u^m_{\l^m}(x)=u^{m-1}_{\l^{m-1}}(y)+\E^m_{\l^m}.$$
First suppose $1<i<m$.  Note that
$$\frac{\partial \E^m_{\l^m}}{\partial x_i} = \sum_{j=1}^{m-1} \frac{\partial \E^m_{\l^m}}{\partial y_j} 
\frac{\partial y_j}{\partial x_i} + e^{x_i-y_{i-1}}-e^{y_i-x_i} -\l_m.$$
On the other hand, by the induction hypothesis and \eqref{dyn2},
$$\frac{\partial u^{m-1}_{\l^{m-1}}(y)}{\partial y_j} 
= -\dot y_j = -\l_m - e^{y_i-x_i}+e^{x_{i+1}-y_i} 
=  - \frac{\partial \E^m_{\l^m}}{\partial y_j}.$$
Thus,
\begin{eqnarray*}
\partial_{x_i} u_\l(x) &=& \sum_{j=1}^{m-1} \frac{\partial u^{m-1}_{\l^{m-1}}(y)}{\partial y_j}  \frac{\partial y_j}{\partial x_i} 
+ \sum_{j=1}^{m-1} \frac{\partial \E^m_{\l^m}}{\partial y_j}  \frac{\partial y_j}{\partial x_i} + e^{x_i-y_{i-1}}-e^{y_i-x_i} -\l_m \\
&=& e^{x_i-y_{i-1}}-e^{y_i-x_i} -\l_m = -\dot x_i,
\end{eqnarray*}
as required.
The cases $i=1$ and $i=m$ are similar.
\end{proof}

We note that, as this is a recursive construction, implicit in the statement of Theorem~\ref{main} is the 
statement that, for each $m\le n$, $x^m(t)$ defines a solution to the $m$-particle Toda flow 
on the iso-spectral manifold corresponding to $\l_1,\ldots,\l_m$.  

Note, in particular, the above shows that $T_t \circ g_\l = g_\l \circ S^\l_t$ on $\T_\l$.
To summarise, if we let $\P_\l=f^{-1}\T_\l$, then the following diagram commutes:

\begin{center}
\begin{tikzpicture}
\node at (1,2) {$\P_\l$};
\draw [->] (1.5,2)--(2.5,2);
\node at (2,2.3) {$f$};
\node at (3,2) {$\T_\l$};
\draw [->] (3.5,2)--(4.5,2);
\node at (4,2.3) {$g_\l$};
\node at (5,2) {$\M_\l$};
\draw [->] (1,1.5)--(1,0.5);
\node at (.5,1) {$R^\l_t$};
\node at (1,0) {$\P_\l$};
\draw [->] (1.5,0)--(2.5,0);
\node at (2,.3) {$f$};
\node at (3,0) {$\T_\l$};
\node at (4,.3) {$g_\l$};
\node at (5,0) {$\M_\l$};
\draw [->] (3.5,0)--(4.5,0);
\draw [->] (3,1.5)--(3,0.5);
\draw [->] (5,1.5)--(5,0.5);
\node at (2.6,1) {$S^\l_t$};
\node at (4.7,1) {$T_t$};
\end{tikzpicture}
\end{center}
In particular, the semi-classical limit of the commutative diagram shown at the end of Section~\ref{scl} is:

\begin{center}
\begin{tikzpicture}
\node at (2.7,2) {$\T_\l$};
\draw [->] (4.5,2)--(3.5,2);
\node at (4,2.3) {$X^*_\l$};
\node at (5.3,2) {$\R^n$};
\node at (2.7,0) {$\T_\l$};
\node at (4,.3) {$\pi$};
\node at (5.3,0) {$\R^n$};
\draw [->] (3.5,0)--(4.5,0);
\draw [->] (2.7,1.5)--(2.7,0.5);
\draw [->] (5.3,1.5)--(5.3,0.5);
\node at (2.3,1) {$S^\l_t$};
\node at (5.7,1) {$T^\l_t$};
\end{tikzpicture}
\end{center}
where now $T^\l_t$ denotes the gradient flow defined by \eqref{giv2} and, with a slight abuse of notation,
$\pi$ denotes the projection $\pi: X=(x^m_i)\mapsto x^n$.

\begin{ex}  Suppose $n=2$ and $x=\l=0$.  Then, on $\T(x)$, 
$$\F_0(X)=\F(X)=e^{-x^1_1}+e^{x^1_1}.$$
This has its unique critical point at $x^1_1=0$, and so
$$X(0)= \begin{array}{ccc} &x^1_1(0)&\\x^2_2(0)&&x^2_1(0)\end{array}
=X_0^*(0,0)=\begin{array}{ccc} &0&\\0&&0\end{array}.$$
Setting $b(0)=f^{-1}(X_0^*(0,0))$, this implies
$$\log\Delta^2_1(b(0))=\log\Delta^2_2(b(0))=\log\Delta^1_1(b(0))=0$$
and hence
$$b(0)=\begin{pmatrix}1&1\\0&1\end{pmatrix}.$$
Now,
$$\epsilon_0=\begin{pmatrix}0&1\\0&0\end{pmatrix},\qquad e^{t\epsilon_0}=\begin{pmatrix}1&t\\0&1\end{pmatrix},$$
and hence
$$b(t)=e^{t\epsilon_0}b(0)=\begin{pmatrix}1&1+t\\0&1\end{pmatrix}.$$
Now, applying the map $f$ again gives
$$X(t)= \begin{array}{ccc} &0&\\-\log(1+t)&&\log(1+t)\end{array},$$
and hence
$$x^2(t)=(\log(1+t),-\log(1+t)).$$
Note that this gives
$$p_1=\frac1{1+t},\qquad p_2=-\frac1{1+t},\qquad q=e^{x^2_2-x^2_1}=\frac1{(1+t)^2}.$$
For the usual Toda lattice this gives the solution
$$\xi(t)=(\log(1+\iota t),-\log(1+\iota t)).$$
The symmetric form of the Lax matrix in this case is given by
$$ \Lambda = \begin{pmatrix} \pi_1 & e^{(\xi_2-\xi_1)/2} \\ e^{(\xi_2-\xi_1)/2} & \pi_2\end{pmatrix} =
\begin{pmatrix} \frac{\iota}{1+\iota t} & \frac1{1+\iota t}\\ \frac1{1+\iota t} & -  \frac{\iota}{1+\iota t}\end{pmatrix}$$
which, at $t=0$, is given by
$$\Lambda(0) = \begin{pmatrix} \iota & 1 \\ 1 & -\iota \end{pmatrix} .$$
\end{ex}
\begin{ex}
Suppose $n=2$, $x=(x,-x)$ and $\l=(\l,-\l)$.  Then
$$X_\l^*(x)=\begin{array}{ccc} &y&\\-x&&x\end{array},$$
where $y\in\R$ is the unique solution to
$$\l+e^{-x-y}=-\l+e^{y-x},$$
that is
$$e^{-y}=\sqrt{\l^2 e^{2x}+1}-\l e^x.$$
Now
$$b(0)=\begin{pmatrix}e^y&e^x\\0&e^{-y}\end{pmatrix},\qquad
\epsilon_\l=\begin{pmatrix}\l&1\\0&-\l\end{pmatrix},$$
$$e^{t\epsilon_\l}=\begin{pmatrix}e^{\l t}&\frac1\l\sinh(\l t)\\0&e^{-\l t}\end{pmatrix},$$
and hence
$$b(t)=e^{t\epsilon_\l} b(0)=\begin{pmatrix}e^{y+\l t} &e^{x+\l t}+e^{-y} \frac1\l\sinh(\l t)\\0&e^{-y-\l t}\end{pmatrix}.$$
The solution is thus given by
$$x^2_1(t)=\log\left( e^{x+\l t}+e^{-y} \frac1\l\sinh(\l t) \right),$$
and $x^2_2(t)=-x^2_1(t)$.  Note that
$$q(t)=e^{x^2_2(t)-x^2_1(t)}=\left( e^{x+\l t}+e^{-y} \frac1\l\sinh(\l t) \right)^{-2}.$$
Note that, when $x\to-\infty$, $y\to 0$ and $b(0)\to I$, and the solution becomes
$$x^2_1(t)=\log\left( \frac1\l\sinh(\l t) \right),$$
as discussed in Example~\ref{exs} below.
\end{ex}

Finally, we consider the flow $R^\l_t$ on $B$ started from the identity.
\begin{thm}\label{main1}
Let $\l\in\R^n$ and $b(t)=e^{t\epsilon_\l}$.  Note that $f(b)=\Pi\eta$, where $\eta(t)=t\l$.
Then $\Pi\eta(t)\in \T_\l$ for $t>0$ and $g_\l(\Pi\eta(t)),\ t>0$ defines a solution to the Toda
flow on $\M_\l$ which is singular at $t=0$. All of the other conclusions of Theorem~\ref{main} 
also hold for $t>0$ with $X(t)=(x^m_i(t))=\Pi\eta(t)$.  
In this case, $b(t)$ is given explicitly by
$b_{ii}(t)=e^{\l_i t}$ and, for $i<j$,
\begin{eqnarray}\label{f-b}
b_{ij}(t) &=& \sum_{i\le k\le j} \left[ \prod_{l=i,\ldots , j;\ l\ne k} (\l_k-\l_l)^{-1}\right] e^{\l_k t}\\
&=&\nn \frac1{2\pi\iota} \oint \frac{e^{tz}dz}{\prod_{i\le k\le j} (z-\l_k)} ,
\end{eqnarray}
where the integration is anti-clockwise around a circle containing $\l_i,\ldots,\l_j$. 
Alternatively, for $i<j$, we can write
\be\label{alt}
b_{ij}(t)=\left[ \prod_{i\le k<l\le j} (\l_k-\l_l)\right]^{-1} 
\begin{vmatrix} e^{\l_i t} &\l_i^{j-i-1}&\dots&\l_i&1\\
e^{\l_{i+1} t} &\l_{i+1}^{j-i-1}&\dots&\l_{i+1}&1\\
&&\vdots&&\\
e^{\l_j t} &\l_j^{j-i-1}&\dots&\l_j&1\end{vmatrix} .
\ee
We recall that, by definition,
$$x^n_1+\cdots+x^n_k=\log\tau_{k},\qquad 1\le k\le n,$$
where $\tau_1=b_{1n}$ and, for $2\le k\le n$,
\be\label{f-tau}
\tau_k=\begin{vmatrix} b_{1,n-k+1} & \dots & b_{1,n-1} & b_{1n} \\
b_{2,n-k+1} & \dots  & b_{2,n-1} & b_{2n}\\
&\vdots&&\\
b_{k,n-k+1} & \dots &b_{k,n-1} &b_{kn} \end{vmatrix}.
\ee
Writing $\tau_1=\tau$, this solution can also be expressed in the more familiar form
\be\label{f1-tau}
\tau_k=\begin{vmatrix} \tau^{(k-1)} & \dots & \tau' & \tau \\
\tau^{(k)} & \dots  & \tau'' & \tau'\\
&\vdots&&\\
\tau^{(2k-2)} & \dots &&\tau^{(k-1)}\end{vmatrix}.
\ee
As matrix integrals, assuming $|\l_i|<1$ for each $i$, 
\be\label{mi-tau}
\tau_{k}=\int_{U(k)} \frac{(\det M)^{k-1}e^{t\mbox{tr} M}}{\prod_{i=1}^n \det (M-\lambda_i I)}  dM.
\ee
When $\l=0$, 
\be\label{f2-tau}
\tau_{k}=\frac{(k-1)!(k-2)!\ldots 1}{(n-1)!(n-2)!\ldots(n-k)!} t^{k(n-k)}.
\ee
\end{thm}
\begin{proof}
Let  $\xi(N)=X^*_\l(-N\rho^n)$, where $\rho^n=(n-1,n-3,\ldots,1-n)$, and set $X^N(t)=S^\l_t\xi(N)$.
By Propositions \ref{S-extend} and \ref{stable}, $X^N(t)=\Pi^{\xi(N)}\eta(t)\in\T_\l$ for each $t>0$.

For $2\le k\le m\le n$, define 
$$r^m_k(N)  = \sum_{i=1}^{k-1} \xi^{m-1}_i(N) - \sum_{i=1}^{k-1} \xi^{m}_i(N).$$
We will show that, as $N\to\infty$, $r^m_i(N)\to+\infty$ for each $2\le i\le m\le n$. 
Note that this implies that $\Pi^{\xi(N)}(\eta)(t)$ converges to $\Pi(\eta)(t)$ for each $t>0$, and hence
that $\Pi(\eta)(t)\in\T_\l$ for each $t>0$.  But $\Pi(\eta)(t)$ also evolves according to \eqref{dyn-rsk} for 
$t>0$, so by Theorem~\ref{main}, $g_\l(\Pi\eta(t)),\ t>0$ defines a solution to the Toda flow on $\M_\l$,
as claimed.  For convenience, write $X^N(0)=X=(x^m_i)$.
For $1\le m\le n$, define $\tilde X=(\tilde x^m_i)\in\T$ by $\tilde x^m=x^m+N\rho^m$.
Note that $\tilde x^n=0$.  Then $\F_\l(X)=e^N\F_{e^{-N}\l}(\tilde X)$ and $\tilde X=X^*_{e^{-N}\l}(0)$.
Thus, as $N\to\infty$, $\tilde X\to X^*_0(0)$.  
It follows that, as $N\to\infty$, for each $1\le i\le m\le n$, $x^m_i\sim-N\rho^m_i$ and 
hence $r^m_k(N)\sim N(k-1)\to\infty$, as required.

We verify the formula \eqref{f-b}
by induction.  Without loss of generality, we only need to show that the formula holds for $b_{1n}$,
$n\ge 2$.  As $b(t)$ is a continuous function of $\l$ we can also
assume, for convenience, that the $\l_i$ are distinct.  Write $b_n(t)=b_{1n}(t)$.  
First we note that
$$b_2(t)=\int_0^t e^{\l_2 s+\l_1(t-s)} ds =\frac1{\l_1-\l_2}(e^{\l_1 t}-e^{\l_2 t}).$$
Now assume that, for each $t>0$,
$$b_{n-1}(t)=\sum_{1\le k\le n-1} \left[ \prod_{l=1,\ldots , n-1;\ l\ne k} (\l_k-\l_l)^{-1}\right] e^{\l_k t}.$$
From the definition \eqref{def-b}, we can write
\begin{eqnarray*}
b_{1n}(t) &=& \int_0^t b_{n-1}(s) e^{\l_n (t-s)}  ds\\
&=& \sum_{k=1}^{n-1} \left[ \prod_{l=1,\ldots , n-1;\ l\ne k} (\l_k-\l_l)^{-1}\right] 
\left[ \int_0^t e^{(\l_k-\l_n) s} ds \right] e^{\l_n t}\\
&=& \sum_{k=1}^{n-1} \left[ \prod_{l=1,\ldots , n;\ l\ne k} (\l_k-\l_l)^{-1}\right] 
\left[ e^{(\l_k-\l_n) t} -1 \right] e^{\l_n t}\\
&=&  \sum_{k=1}^{n-1} \left[ \prod_{l=1,\ldots , n;\ l\ne k} (\l_k-\l_l)^{-1}\right] 
\left[ e^{\l_kt}-e^{\l_n t}  \right] .
\end{eqnarray*}
It therefore suffices to show that
$$\sum_{k=1}^{n-1} \left[ \prod_{l=1,\ldots , n;\ l\ne k} (\l_k-\l_l)^{-1}\right] = - \prod_{1\le k\le n-1} (\l_n-\l_k)^{-1}$$
or, equivalently,
\be\label{wts}
\sum_{k=1}^n \left[ \prod_{l=1,\ldots , n;\ l\ne k} (\l_k-\l_l)^{-1}\right]=0.
\ee
To see that this holds, denote 
$$\Delta_m(a_1,\ldots,a_m)=\prod_{1\le i<j \le m} (a_i-a_j)=\det\left[ a_i^{n-j} \right]_{i,j=1,\ldots,m}$$ 
and note that
\begin{align*}
 \Delta_n(\l) \sum_{k=1}^n  \left[ \prod_{l=1,\ldots , n;\ l\ne k}  (\l_k-\l_l)^{-1}\right] 
&= \sum_{i=1}^n (-1)^{i-1} \Delta_{n-1}(\l_1,\ldots,\widehat{\l_i},\ldots,\l_n) \\
&= \begin{vmatrix} 1&\l_1^{n-2}&\dots&\l_1&1\\
1&\l_2^{n-2}&\dots&\l_2&1\\
&&\vdots&&\\
1&\l_n^{n-2}&\dots&\l_n&1\end{vmatrix} =0,
\end{align*}
which implies \eqref{wts}.  Note that essentially the same calculation yields the alternative formula \eqref{alt}.

Set $\tau_1=b_{1n}$ and, for $2\le k\le n$,
$$\tau_k=\det \left[ b_{ij}\right]_{\ 1\le i\le k,\ m-k+1\le j\le m}.$$
Then
$$x^n_1+\cdots+x^n_k=\log\tau_k,\qquad 1\le k\le n,$$
and since $x^n$ satisfies the Toda equations \eqref{toda1}, this implies
$$(\log\tau_k)''=-\frac{\tau_{k+1}\tau_{k-1}}{\tau_k^2}$$
for each $1\le k\le n$ with the conventions $\tau_0=1$ and $\tau_{n+1}=0$.

On the other hand, the tau functions $\tilde \tau_k$ defined by $\tilde\tau_0=1$,
$\tilde\tau_1=\tau_1$,
$$\tilde\tau_k = \det\left[\tau_1^{(k+i-j-1)} \right]_{1\le i,j\le k},\qquad k\ge 2$$
also satisfy
$$(\log\tilde \tau_k)''=-\frac{\tilde\tau_{k+1}\tilde\tau_{k-1}}{\tilde\tau_k^2}.$$
This is well-known and is easily verified using basic properties of Wronskians
and Sylvester's identity.  It follows, by induction, that $\tilde\tau_k=\tau_k$ for $1\le k\le n$, 
and in fact $\tilde\tau_k=0$ for $k>n$.

To obtain the last formula \eqref{f2-tau}, note that when $\l=0$, $b_{ij}(t)=t^{j-i}/(j-i)!$ for $i<j$, hence
\begin{eqnarray*}
\tau_k &=& \det\left[\frac{t^{n-k+j-i}}{(n-k+j-i)!} 1_{n-k+j-i\ge 0} \right]_{i,j=1,\ldots,k} \\
&=& \frac{(k-1)!(k-2)!\ldots 1}{(n-1)!(n-2)!\ldots(n-k)!} t^{k(n-k)} 
\det\left[\binom{n-i}{k-j} \right]_{i,j=1,\ldots,k} ,
\end{eqnarray*}
with the convention that $\binom{a}{b}=0$ if $b>a$.  
Now, by a theorem of Gessel and Viennot~\cite{gv},
$$\det\left[\binom{n-i}{k-j} \right]_{i,j=1,\ldots,k}=1,$$
so we are done.
\end{proof}

\begin{ex}\label{exs}
Suppose $n=2$ and $\l=(\l,-\l)$. Then
$$x(t)=\left(\log\left[\frac1{\l}\sinh(\l t)\right],-\log\left[\frac1\l\sinh(\l t)\right]\right).$$
This yields the solution
$$\xi(t)=\left(\log\left[\frac{\iota}\l \sin(\l t)\right],-\log\left[\frac{\iota}\l \sin(\l t)\right]\right)$$
of the usual Toda lattice.  Note that
$$e^{\xi_2-\xi_1}=-\frac{\l^2}{\sin^2(\l t)}.$$
We can also take $\l$ to purely imaginary, $\l=\iota\gamma$ say, where $\gamma\in\R$.  
Then 
$$e^{\xi_2-\xi_1}=-\frac{\gamma^2}{\sinh^2(\gamma t)}.$$
This particular singular solution of the usual Toda lattice is discussed, for example, in~\cite{ck}.
When $\l=0$, 
$$x(t)=(\log t,-\log t),\qquad \xi(t)=(\iota\pi/2+\log t,-\iota\pi/2-\log t),$$
and $$e^{\xi_2-\xi_1} = -1/t^2.$$
\end{ex}


\begin{thebibliography}{10}

\bibitem{bdj}
J.~ Baik, P.~Deift and K.~Johansson.
\newblock On the distribution of the length of the longest increasing subsequence of random permutations.
\newblock {\it J. Amer. Math. Soc.} 12 (1999) 1119--1178.

\bibitem{b}
Y. Baryshnikov. GUEs and queues.
{\em Probab. Th. Rel. Fields} 119 (2001) 256--274.

\bibitem{boc}
F. Baudoin and N. O'Connell.
Exponential functionals of Brownian motion and class one Whittaker functions.
{\em Ann. Inst. H. Poincar\'e Probab. Statist.} 47 (2011) 1096-1120.

\bibitem{bfz} A.  Berenstein, S. Fomin and  A. Zelevinsky. 
Parameterizations of canonical bases and totally positive matrices. 
{\em Adv. Math.} 122  (1996) 49--149.

\bibitem{bbo1} Ph. Biane, Ph. Bougerol and N. O'Connell.
Littelmann paths and Brownian paths. {\em Duke Math. J.}  130  (2005) 127--167. 

\bibitem{bbo2} Ph. Biane, Ph. Bougerol and N. O'Connell.
Continuous crystals and Duistermaat-Heckman measure for Coxeter groups.
{\em Adv. Math.} 221 (2009) 1522--1583.

\bibitem{bc} A. Borodin and I. Corwin.  Macdonald processes. 
{\em Probab. Th. Rel. Fields} 158 (2014) 225--400.

\bibitem{bf} A. Borodin and P. Ferrari. 
Large time asymptotics of growth models on space-like paths. I. PushASEP. 
{\em Electron. J. Probab.} 13 (2008) 1380--1418.

\bibitem{bp} A. Borodin and L. Petrov.
Nearest neighbor Markov dynamics on Macdonald processes.
arXiv:1305.5501
	
\bibitem{bj}  Ph. Bougerol and Th. Jeulin. Paths in Weyl chambers and random matrices.
{\em Probab. Th. Rel. Fields}  124  (2002) 517--543.

\bibitem{ck} L. Casian and Y. Kodama.
Singular structure of Toda lattices and cohomology of certain compact Lie groups.
{\em J. Comput. Appl. Math.} 202 (2007) 56--79.

\bibitem{ch} R. Chhaibi.  {\em Mod\`ele de Littelmann pour cristaux g\'eom\'etriques, fonctions de Whittaker sur 
des groupes de Lie et mouvement brownien.}  PhD thesis, Universit\'e Paris VI - Pierre et Marie Curie, 2012.

\bibitem{cosz} I. Corwin, N. O'Connell, T. Sepp\"al\"ainen and N. Zygouras. 
Tropical combinatorics and Whittaker functions. 
{\em Duke Math. J.}, in press. 

\bibitem{df} P. Diaconis and J. Fill. 
Strong Stationary Times Via a New Form of Duality. 
{\em Ann. Probab.} 18 (1990) 1483--1522.

\bibitem{eklp} N. Elkies, G. Kuperberg, M. Larsen and J. Propp.
Alternating-sign matrices and domino tilings. II.  {\em J. Algebraic Combin.} 1 (1992), no. 3, 219--234.

\bibitem{efs} N.M. Ercolani, H. Flaschka and S. Singer.
The Geometry of the Full Toda Lattice. {\em Progress in Mathematics} 115 (1993) 181--226.

\bibitem{fz} S. Fomin and  A. Zelevinsky. 
Double Bruhat cells and total positivity. 
{\em J. Amer. Math. Soc.} 12 (1999) 335--380.

\bibitem{gklo} A. Gerasimov, S. Kharchev, D. Lebedev and S. Oblezin.
On a Gauss-Givental representation of quantum Toda chain wave equation.
{\em Int. Math. Res. Not.} (2006) 1--23.

\bibitem{gv} I. Gessel and G. Viennot.
Binomial determinants, paths, and hook length formulae.
{\em Adv. Math} 58 (1985) 300--321.

\bibitem{g} A. Givental. Stationary phase integrals, quantum Toda lattices, 
flag manifolds and the mirror conjecture. {\em Topics in Singularity Theory},
AMS Transl. Ser. 2, vol. 180, AMS, Rhode Island (1997) 103--115.

\bibitem{gk} A. Givental and B. Kim.
Quantum cohomology of flag manifolds and Toda lattices. 
{\em Commun. Math. Phys.} 168 (1995) 609--641.

 \bibitem{gtw}
     J. Gravner, C.A. Tracy and
    H.  Widom. Limit theorems for height fluctuations in a class
    of discrete space and time growth models. {\em J. Stat. Phys.}
     102 (2001) 1085--1132.

\bibitem{is} T. Ishii and E. Stade. New formulas for Whittaker functions on $GL(n,\R)$.
{\em J. Fun. Anal.} 244 (2007) 289--314.

\bibitem{jk} D. Joe and B. Kim.
Equivariant mirrors and the Virasoro conjecture for flag manifolds.
{\em Int. Math. Res. Notices} 15 (2003) 859--882.

\bibitem{joc} L. Jones and N. O'Connell.
Weyl Chambers, Symmetric Spaces and Number Variance Saturation. 
{\em ALEA} 2 (2006) 91--118

\bibitem{ki} A. N. Kirillov.  Introduction to tropical combinatorics. {\em Physics and Combinatorics. Proc.
Nagoya 2000 2nd Internat.Workshop} (A. N. Kirillov and N. Liskova, eds.), World Scientific,
Singapore, 2001, pp. 82--150.

\bibitem{ks} Y. Kodama and B.A. Shipman. 
The finite non-periodic Toda lattice: a geometric and topological viewpoint.
arXiv:0805.1389
{\em Modern Encyclopedia of Mathematical Physics (MEMPhys)}, to appear.

\bibitem{ky} Y. Kodama and J. Ye.
Toda lattices with indefinite metric II: Topology of the iso-spectral manifolds.
{\em Physica D}, 121 (1998) 89--108.

\bibitem{k} B. Kostant.  Quantisation and representation theory.  In: {\em Representation Theory
of Lie Groups}, Proc. SRC/LMS Research Symposium, Oxford 1977, LMS Lecture Notes 34,
Cambridge University Press, 1977, pp. 287--316.

\bibitem{k1} B. Kostant.
The Solution to a Generalized Toda Lattice and Representation Theory.
{\em Adv. Math.} 34 (1979) 195--338.

\bibitem{llp}
C.~Lecouvey, E.~Lesigne, and M.~Peign\'e.
\newblock {Random walks in Weyl chambers and crystals.}
\newblock {\em Proc. Lond. Math. Soc.}, 104:323--358, 2012.

\bibitem{my}  H. Matsumoto and M. Yor.
A version of Pitman's $2M-X$ theorem for geometric Brownian motions.
{\em C. R. Acad. Sci. Paris} 328 (1999) 1067--1074.

\bibitem{n} E. Nordenstam.
On the Shuffling Algorithm for Domino Tilings.
{\em Electron. J. Probab.} 15 (2010) 75--95.

\bibitem{ny} M. Noumi and Y. Yamada. 
Tropical Robinson-Schensted-Knuth correspondence and birational Weyl group actions.
{\em Adv. Stud. Pure Math.} 40 (2004) 371--442.

\bibitem{oc03}
N. O'Connell.  Random matrices, non-colliding processes and queues. 
{\em S\'eminaire de Probabilit\'es}, XXXVI, 165--182, 
Lecture Notes in Math., 1801, Springer, 2003.

\bibitem{oc03a}
N. O'Connell. 
A path-transformation for random walks and the Robinson-Schensted correspondence.
{\em Trans. Amer. Math. Soc.} 355 (2003) 3669--3697.

\bibitem{oc03b}
N.~O'Connell.
\newblock Conditioned random walks and the {RSK} correspondence.
\newblock {\em J. Phys. A}, 36(12):3049--3066, March 2003.

\bibitem{noc}
N. O'Connell.  Directed polymers and the quantum Toda lattice.
{\em Ann. Probab.} 40 (2012) 437--458.

\bibitem{noc1}
N.~O'Connell. Whittaker functions and related stochastic processes.
To appear in proceedings of Fall 2010 MSRI semester 
{\em Random matrices, interacting particle systems and integrable systems}.
arXiv:1201.4849

\bibitem{op} N. O'Connell and Y. Pei.
A $q$-weighted version of the Robinson-Schensted algorithm.
arXiv:1212.6716

\bibitem{osz} N. O'Connell, T. Sepp\"al\"ainen and N. Zygouras. 
Geometric RSK correspondence, Whittaker functions and symmetrized random polymers.
{\em Inventiones Math.}, 2013.  

\bibitem{oy} N. O'Connell and M. Yor.
A representation for non-colliding random walks.
{\em Elect. Comm. Probab.} 7 (2002).

\bibitem{pei} Y. Pei.  A symmetry property of $q$-weighted Robinson-Schensted and 
other branching insertion algorithms. arXiv:1306.2208

\bibitem{pitman}
J.W. Pitman. One-dimensional Brownian motion and the three-dimensional
Bessel process. { Adv. Appl. Probab.} 7 (1975) 511-526.

\bibitem{p} J. Propp. Generalized domino-shuffling.  
{\em Theoret. Comput. Sci.} 303 (2003) 267-301.

\bibitem{r} K. Rietsch.
A mirror construction for the totally nonnegative part of the Peterson variety.
{\em Nagoya Math. J.} 183 (2006) 105--142.

\bibitem{w} J. Warren.
Dyson's Brownian motions, intertwining and interlacing.
{\em Electron. J. Probab.}  12  (2007) 573--590.

\bibitem{ww}  J. Warren and P. Windridge. 
Some examples of dynamics for Gelfand-Tsetlin patterns. 
{\em Electron. J. Probab.} 14 1745--1769.

\end{thebibliography}
\end{document}